\documentclass[onefignum,onetabnum, dvipsnames]{siamonline190516}

\usepackage{upgreek}
\usepackage{amsfonts, amsmath, amssymb}
\usepackage{graphicx}
\usepackage{overpic}
\usepackage{epstopdf}
\usepackage{algorithmic}
\usepackage{enumerate}
\usepackage{enumitem}
\usepackage{bm}
\ifpdf
  \DeclareGraphicsExtensions{.eps,.pdf,.png,.jpg}
\else
  \DeclareGraphicsExtensions{.eps}
\fi

\graphicspath{{image/}}

\usepackage{verbatim} 
\usepackage{makecell}
\usepackage{pifont}
\newcommand{\cmark}{\ding{51}}%
\newcommand{\xmark}{\ding{55}}
\usepackage{enumitem}
\setlist[enumerate]{leftmargin=.5in}
\setlist[itemize]{leftmargin=.5in}

\newsiamthm{claim}{Claim}
\newsiamremark{conjecture}{Conjecture}
\newsiamremark{remark}{Remark}
\newsiamremark{example}{Example}
\newsiamremark{hypothesis}{Hypothesis}
\newsiamremark{problem}{Problem}
\newsiamremark{assumption}{Assumption}

\usepackage{etoolbox}
\AtEndEnvironment{remark}{\null\hfill$\Diamond$}
\AtEndEnvironment{assumption}{\null\hfill$\Diamond$}

\newcommand{\mcl}{\mathcal}

\newcommand{\mbf}{\mathbf}
\newcommand{\mbb}{\mathbb}

\newcommand{\Div}{\text{div}}

\newcommand{\dd}{{\rm d}}

\newcommand{\bs}{\mbf s}
\newcommand{\bx}{\mbf x}

\newcommand{\bz}{\mbf z}
\newcommand{\by}{\mbf y}

\newcommand{\bw}{\mbf w}

\newcommand{\veps}{\varepsilon}

\newcommand{\bphi}{{\bm{\phi}}}

\newcommand{\btheta}{\bm{\theta}}

\newcommand{\mL}{\mcl{L}}

\newcommand{\mB}{\mcl{B}}

\newcommand{\mU}{\mcl{U}}

\newcommand{\mP}{\mcl{P}}

\newcommand{\mT}{\mcl{T}}

\newcommand{\R}{\mbb{R}}

\newcommand{\st}{{\rm\,s.t.}}

\newcommand{\M}{\mathcal{M}}

\newcommand{\hF}{\overline{F}}

\definecolor{darkred}{rgb}{.7,0,0}

\definecolor{darkgreen}{rgb}{.15,.55,0}

\definecolor{darkblue}{rgb}{0,0,0.7}

\usepackage{amsopn}

\DeclareMathOperator*{\argmin}{arg\,min}

\DeclareMathOperator*{\minimize}{{\rm minimize}}

  \AtEndEnvironment{example}{\null\hfill$\Diamond$}
  \AtEndEnvironment{runningexample}{\null\hfill$\Diamond$}

\reversemarginpar
\usepackage[colorinlistoftodos,prependcaption,textsize=tiny]{todonotes}

\headers{Error estimates for solving nonlinear {PDEs} with {kernels and GPs}}{P. Batlle, Y. Chen, B. Hosseini, H. Owhadi, AM. Stuart}

\title{Error Analysis of 
Kernel/GP Methods for Nonlinear and Parametric PDEs}

\author{Pau Batlle\thanks{Computing and Mathematical Sciences, Caltech, Pasadena, CA
    (\email{pbatllef@caltech.edu}, \email{yifanc@caltech.edu},
    \email{owhadi@caltech.edu}, \email{astuart@caltech.edu}).}
  \and Yifan Chen\footnotemark[1]
  \and Bamdad Hosseini\thanks{Department of Applied Mathematics, University of Washington, Seattle, WA
  (\email{bamdadh@uw.edu})}
  \and Houman Owhadi\footnotemark[1]
  \and  Andrew M Stuart\footnotemark[1]
}

\ifpdf
\hypersetup{
  pdftitle={GP-PDE},
  pdfauthor={P Batlle, Y. Chen, B. Hosseini, H. Owhadi and A. Stuart}
}
\fi

\begin{document}

\maketitle

\begin{abstract}
We introduce a priori Sobolev-space error estimates for the solution of
nonlinear, and possibly parametric, PDEs using Gaussian process and kernel based methods.
The primary assumptions are: (1) a continuous embedding of the reproducing kernel Hilbert space of the kernel into a Sobolev space of sufficient regularity; and (2) the stability of the 
differential operator and the solution map of the PDE
between corresponding Sobolev spaces. The proof is articulated around Sobolev norm error estimates for kernel interpolants and relies on the minimizing norm property of the solution. The error estimates demonstrate dimension-benign convergence rates if the solution space of the PDE is smooth enough. We illustrate these points with applications to high-dimensional nonlinear elliptic PDEs and parametric PDEs. 
 Although some recent machine learning methods have been presented as breaking the curse of dimensionality in solving 
 high-dimensional PDEs, our analysis suggests a more nuanced picture: there is a trade-off between the regularity of the solution and the presence of the curse of dimensionality. Therefore, our results are in line 
 with the understanding that the curse is absent when the solution is regular enough. 
\end{abstract}




\section{Introduction}\label{sec:introduction}
In recent years the adoption of machine learning in the natural sciences and engineering 
has led to the development of new methods for solving PDEs 
\cite{raissi2019physics, weinan2017deep, weinan2018deep, li2020fourier, Lu2021}. The majority of these 
methods rely on the approximation power of artificial neural networks (ANNs)
either as a function class to approximate the solution of the PDE 
or as a high-dimensional function class to approximate the solution map of the PDE. 
Despite the empirical success of the aforementioned ANN based methods, current 
theoretical understanding of these PDE solvers is  scarce and, beyond particular PDEs (e.g., \cite{shin2020convergence, lu2021machine, de2022error}),  results are oftentimes limited to existence results rather than convergence guarantees or rates.

Similar to ANNs, kernel methods and Gaussian processes (GPs) have been very effective in scientific computing and machine learning \cite{scholkopf2018learning, muandet2017kernel, berlinet2011reproducing, williams1996gaussian} and at the same time they are supported by rigorous theoretical foundation \cite{berlinet2011reproducing, wendland2004scattered, owhadi2019operator}. Recently in \cite{CHEN2021110668}, the authors introduced a kernel collocation method for solving arbitrary nonlinear PDEs with a rigorous convergence guarantee. The theory presented in that work 
was based on the assumptions that (1) the solution belongs to the 
reproducing kernel Hilbert space (RKHS) defined by the underlying kernel 
which in turn is embedded in the Sobolev space $H^s$ for $s>d/2 \: +$``order of the PDE'' (where $d$ is the dimension of the domain of the PDE) and (2) the fill-distance between collocation points goes to zero. Convergence was proved via a compactness argument but 
no convergence rates were provided. 

The goal of this article is to provide quantitative convergence rates for the 
PDE solver introduced in \cite{CHEN2021110668}. Our quantitaitve rates 
also reveal the interplay between the regularity of the solution of the PDE and the dimension 
$d$ of the problem. At the same time we make improvements 
to the methodology of \cite{CHEN2021110668} and extend it to the case of parametric PDEs. In the rest of this section we summarize our main contributions in \Cref{sec:contributions} followed by a review of the relevant literature in \Cref{sec:relevant-lit}, and 
an outline of the article in \Cref{sec:outline}.


\subsection{Our Contributions}\label{sec:contributions}
Throughout the article we consider parametric PDEs of the form 
\begin{equation}
  \label{abstract-parametric-PDE-intro}
  \left\{
    \begin{aligned}
      \mP(u^\star) (\bx; \btheta) & = f(\bx; \btheta), && (\bx, \btheta) \in
      \Omega \times \Theta, \\
      \mB(u^\star) (\bx; \btheta) & = g(\bx; \btheta), && (\bx, \btheta) \in
      \partial \Omega \times \Theta,
    \end{aligned}
    \right.
\end{equation}
where $\Omega \subset \R^d$ is a bounded connected domain with an 
appropriately smooth boundary 
$\partial \Omega$, $\mP$ and $\mB$ are the interior and boundary differential operators 
that define the PDE and $f, g$ are the source and boundary data.  $\bx$ denotes the 
spatial variable with $\btheta$ denoting a parameter belonging to a compact 
set $\Theta \subset \R^p$. The function $u^\star$ denotes the exact, strong solution 
of this PDE.

We view the solution $u^\star$ as a function on $\overline{\Omega} \times \Theta$
and approximate it in an appropriate RKHS by imposing the PDE as a constraint on a
set of collocation points 
in the product space $\overline{\Omega} \times \Theta$. 
Our main contributions are three-fold as summarized below:

\begin{enumerate}
    \item We extend the kernel PDE solver of \cite{CHEN2021110668} to the case of the
    parametric PDE \eqref{abstract-parametric-PDE-intro}. 
    This extension follows by viewing the solution  $u^\star(\bx; \theta)$ as a continuous function 
    defined on $\overline{\Omega} \times \Theta$ and approximating it with a function $u^\dagger$ in an appropriate RKHS $\mU$
    after imposing the PDE as a constraint on a set of collocation points. 
    At the same time we improve the efficacy and performance  of the Gauss-Newton (GN) algorithm of \cite{CHEN2021110668} through an approach of ``linearize first then apply the kernel solver''. For many prototypical PDEs this new approach leads to smaller kernel matrices that can be factored or inverted more efficiently.
    These numerical strategies are 
    outlined in \Cref{sec:review}.
    
    \item  We provide explicit a priori convergence rates for the kernel estimator $u^\dagger \in \mU$ 
    to the true solution $u^\star$. Our proof relies on three assumptions: (1) the RKHS 
    $\mU \subset H^s(\Omega)$ for $s > (d+p)/2 + \text{``order of the PDE''}$; (2) the true 
    unique solution $u^\star \in \mU$; (3) the forward PDE operator and the associated 
    solution map of the PDE are Lipschitz stable. Our error estimates are of the general form 
    \begin{equation}\label{abstract-rate}
        \| u^\dagger - u^\star \|_{L^2(\Omega)} \lesssim h^{s} \| u^\star \|_{\mU}, 
    \end{equation}
    where $h$ is the fill-distance (mesh-norm) of our collocation points. Indeed, if expressed 
    in terms of $N$, the number of collocation points, the rate will read as $\mathcal{O}(N^{-s/(d+p)})$.
    The above rate indicates a trade-off between the regularity of the solution space and the dimension $d+p$ of $\Omega \times \Theta$ stating that the convergence rate is dimension-benign 
    so long as the solution $u^\star$ is sufficiently regular; these results are outlined in \Cref{sec:error-bounds}.


         \item In fact, our method for proving the rate \eqref{abstract-rate} is more general 
    than the case of 
    PDEs (see \Cref{fig:proof-steps} for the road map of the proof technique).  
    The proof can be viewed as a  recipe for convergence analysis of solutions to 
    nonlinear functional equations of the form $\mP(u) = f$ where $u,f$ belong to sufficiently 
    regular function spaces and $\mP$ is invertible (at least locally). 
    Then the Lipschitz 
    stability of $\mP$ and $\mP^{-1}$ plus RKHS interpolation bounds on $f$ 
    yield convergence rates for $u$. Results at this level of generality are presented  in \cite{schaback2016all}
    for linear maps $\mP$ which are then extended to nonlinear problems in \cite{bohmer2013nonlinear,bohmer2020nonlinear}. In these works, the stability of the discretization method is furthermore assumed. In the GP methodology, this property is guaranteed, due to the minimal RKHS norm property of the solution. This has been pointed out in \cite[Sec. 10]{schaback2016all} for linear PDEs. Our theory can be seen as a generalization of the result in \cite{schaback2016all} to the nonlinear case.

    \item We present a suite of numerical experiments that elucidate and extend our theoretical 
    analysis in item 2. We present an example of a nonlinear elliptic PDE with a prescribed 
    solution of varying regularity in various dimensions. We then 
    explore the interplay between regularity and dimensionality as well as the 
    rate in \eqref{abstract-rate}.
    We further verify our result for a one dimensional parametric PDE by varying $p$, the dimension 
    of the parameter space $\Theta$. 
    Because of this trade-off between regularity and dimensionality, 
    showing that a numerical method remains accurate for a high dimensional PDE may not be an indication that it is breaking the curse of dimensionality but simply an indication that the problem being solved is very regular; see our experiments in \Cref{sec:numerics} and 
    in particular \Cref{section: High D HJB eqn}.


\end{enumerate}

\begin{figure}
    \centering
    \includegraphics[clip, trim = {0cm 2.5cm 10cm 2cm}, width = .8 \textwidth]{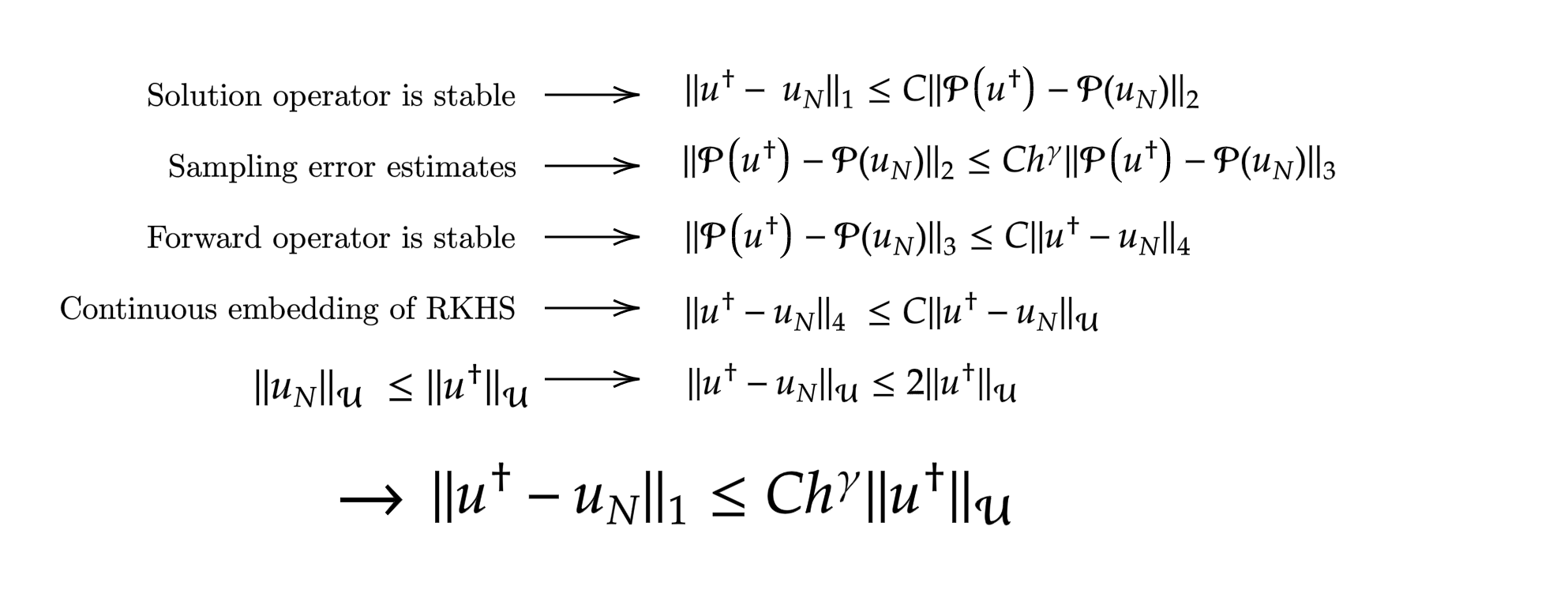}
    \caption{A summary of the main steps in our proof of convergence rates 
outlined in \Cref{thm:abstract-error-analysis,thm:PDE-error-bound-no-BC,thm:PDE-error-bound-with-BC,thm:parametric-PDE-error-bound}. The 1--4 norms denote 
arbitrary norms on appropriate Banach spaces while the $\|\cdot\|_{\mathcal{U}}$-norm  
can be chosen as an RKHS norm or another desired norm with respect to which the numerical algorithm is stable.}
    \label{fig:proof-steps}
\end{figure}

\subsection{Relevant Literature}\label{sec:relevant-lit}
Below we present a brief review of the relevant literature to the current work. 

\subsubsection{Kernel and Gaussian Process Solvers for PDEs}
As mentioned earlier our algorithmic and theoretical developments are focused on 
the kernel method introduced in \cite{CHEN2021110668} and extending that approach to parametric PDEs. Further extensions and applications 
of the aforementioned framework can 
also be found in \cite{mou2022numerical, meng2022sparse, long2022kernel, chen2023sparse}. When applied to  linear PDEs our kernel method coincides 
with the so-called symmetric collocation method \cite[Sec.~14]{schaback2006kernel} and is closely 
associated with  radial basis function (RBF)  PDE solvers \cite{fornberg2015solving, franke1998solving}.  
Various error analyses for RBF  collocation methods  can be found in
\cite{franke1998convergence,franke1998solving}.
 In particular, the article   \cite{giesl2007meshless} is the closest to our work and their rates
 coincide with ours in the linear PDE setting.
The article \cite{cheung2018h} presents similar bounds for the so-called Kansa method 
\cite{kansa1990multiquadrics-I, kansa1990multiquadrics-II}, a non-symmetric 
RBF collocation PDE solver.  Finally, \cite{schaback2016all} presents 
an abstract set of convergence rates for RBF interpolation of ``well-posed'' linear maps between 
regular function spaces that includes RBF PDE solvers as a special case. 
All of the aforementioned analyses consider linear PDEs and some generalizations to nonlinear problems are studied in \cite{bohmer2013nonlinear,bohmer2020nonlinear}. 

The deep connection of  Kernels and RKHSs to the theory of GPs 
\cite{bogachev1998gaussian, larkin1972gaussian, williams1996gaussian, van2008reproducing}
suggests that kernel PDE solvers can be viewed from lens of probability theory 
as a conditioning problem for GPs. While not as extensively developed as the 
kernel solvers mentioned earlier, this direction has been explored for the solution of linear PDEs
as well as nonlinear ODEs
\cite{chkrebtii2016bayesian, cockayne2019bayesian, Owhadi:2014, sarkka2011linear, swiler2020survey} and recent 
works have extended this idea to  some nonlinear and time-dependent PDEs 
\cite{cockayne2016probabilistic, raissi2018numerical, wang2021bayesian}.
The GP interpretation is attractive due to the ability to provide rigorous uncertainty 
estimates along with the solution to the PDE. The idea here is that the uncertainties can serve 
as a posterior or a priori error indicators for the PDE solver. Some ideas related to 
this direction were discussed in \cite{CHEN2021110668, cockayne2016probabilistic}.
A fully probabilistic GP interpretation of our kernel framework
for linear PDEs can be found in \cite{Owhadi:2014,cockayne2016probabilistic, owhadi2019operator} but the case of 
nonlinear PDEs remains 
partially investigated \cite{CHEN2021110668,meng2022sparse,mou2022numerical,long2022autoip}.
Moreover we note that in the GP framework, hierarchical Bayes learning can be used to select kernels to get better convergence rates \cite{chen2020consistency,wilson2016deep,owhadi2019kernel,darcy2023one}.

\subsubsection{Parametric and High-dimensional PDEs}
Parametric PDEs are ubiquitous is physical sciences and engineering
and in particular in the context of uncertainty quantification (UQ)  and solution of stochastic PDEs (SPDEs) \cite{cialenco2012approximation, ghanem2003stochastic, le2010spectral, ye2013kernel}.
A vast literature exists on the subject, connecting it 
to reduced basis models \cite{almroth1978automatic, noor1980reduced}, 
emulation of computer codes \cite{kennedy2001bayesian},  reduced order 
models \cite{lucia2004reduced}, and numerical homogenization \cite{owhadi2019operator};
for settings that most closely resemble our problems we
refer the reader to \cite{ghanem2003stochastic, xiu2010numerical, cohen2015approximation} 
for a general overview. 
Broadly speaking, the dominant approaches for approximation of high-dimensional and 
parametric solution maps include
polynomial/Taylor approximation methods \cite{beck2012optimal, Chkifa2012, chkifa2014high, nobile2008sparse, nobile2008anisotropic};
Galerkin methods \cite{gunzburger2014stochastic, cohen2010convergence};
 reduced basis methods \cite{hesthaven2016certified}; and more recently 
 ANN operator learning techniques such as \cite{li2020fourier, Lu2021}. 
In comparison to the aforementioned works we propose to 
directly approximate the solution of the parametric PDE as a function 
on the tensor product space of the physical and parameter domains 
in a similar spirit as \cite{kempf2019kernel}. The recent article \cite{batlle2023kernel} also presents a kernel based operator learning 
approach to various PDE problems including parametric PDEs.



\subsubsection{The Curse of Dimensionality}







Although the trade-off
between regularity and accuracy is well understood in numerical approximation/integration, where it has led to the development of the Kolmogorov $N$-width and stress tests for finite-element methods \cite{pinkus2012n,melenk2000n,babuvska2000can},  its impact is oftentimes overlooked when communicating the convergence of Machine Learning  and Deep Learning methods for high-dimensional PDEs. In particular, since artificial neural networks (ANNs) can be interpreted as kernel methods \cite{neal1996priors,lee2017deep,owhadi2020ideas} with data-dependent parameterized kernels, our results raise the further question of understanding whether the (empirically observed) convergence of ANN-based methods for high-dimensional PDEs is an indication of the absence of the curse (i.e., the regularity of the solution in selected numerical experiments is high) or the breaking of that curse. 
In particular, empirically observing numerical accuracy for an algorithm and particular solutions is insufficient to prove that the curse of dimensionality is broken, and one must also show that the underlying problem and those solutions are not too regular. We emphasize that the curse of dimensionality referred to here, is the one associated with the worsening of the accuracy of a numerical approximation algorithm as a function of the dimension of the domain of the PDE as opposed to the impact of the curse on the number of degrees of freedom in the implementation level (e.g., finite difference methods suffer from that second curse but ANN/kernel based methods do not).

\subsection{On the Importance of Developing and Benchmarking against Kernel/GP Methods}

The proposed work aims to further develop Gaussian Process (GP) and kernel methods for solving PDEs. 
We are motivated to do so because GP methods have the potential to offer the best of both worlds by combining the strengths and flexibility of traditional and Deep Learning (DL) methods while also being equipped with
theoretical guarantees and uncertainty quantification (UQ) capabilities. 

Compared to traditional methods (such as finite element methods (FEM), finite volume methods (FVM), finite difference methods (FDM), 
spectral methods, etc), GP methods generalize  meshless, RBF, optimal recovery methods and are flexible and applicable in high dimensions. Compared to DL methods that use an expressive neural network representation, GPs offer transparent methods that are easy to reproduce and analyze.
Furthermore the natural probabilistic interpretation of GPs enables convenient UQ
and also facilitates the process of scientific discovery itself \cite{owhadi2019operator};
FEM and DL methods do not interface so cleanly with UQ. Moreover, with hierarchical kernel learning \cite{chen2020consistency,wilson2016deep,owhadi2019kernel,darcy2023one}, GP methods can also be made highly expressive. In fact, as shown in the table below, GP methods can offer many advantages over traditional and DL methods. 
In the context of PDEs
these advantages include: greater flexibility, applicability in high dimensions, provable guarantees, near-linear complexity computation, Occam's razor principle in the design of statistical models, mathematical transparency and interpretability, and ease of reproducibility; see \Cref{tab:GPs-are-the-best}. Although software support for GPs is currently not as advanced as that for DL and traditional methods, GPs are still easy to program and can be seamlessly integrated into an engineering pipeline.

\begin{table}[htp]
\begin{tiny}
\begin{center}
\begin{tabular}{ |c||c|c|c|c|c|c|c|c|c| }
 \hline 
 Method & \makecell{Ease of \\ implementation \\ in high-dimensions} 
 & \makecell{Provable\\ guarantees} & \makecell{Near \\ linear \\ complexity}& \makecell{Occam's\\ razor} & Transparent & \makecell{Ease of\\ reproducibility} & \makecell{Built-in \\ UQ} & \makecell{Software\\ support}\\
 \hline
 Trad. & \xmark 
 & \cmark & \cmark & \cmark & \cmark & \cmark & \xmark & \cmark \\
 \hline
 Kernel & \cmark 
 & \cmark & \cmark & \cmark & \cmark & \cmark & \cmark & Limited\\
  \hline
 ANN & \cmark 
 & Limited & \xmark & \xmark & \xmark & Limited & \xmark  & \cmark\\
 \hline
\end{tabular}
\end{center}
\end{tiny}
\vspace{1ex}
    \caption{A qualitative comparison of the  properties of traditional 
    PDE solvers (such as FEM, FVM, FDM, spectral methods, etc) against kernel methods and ANNs.}
    \label{tab:GPs-are-the-best}
\end{table}

Given their long training times, ANN-based methods may not be competitive with FEM in low dimensions \cite{grossmann2023can}. In contrast, GP-based methods can 
achieve near-linear complexity when combined with fast algorithms for kernel methods such as the sparse Cholesky factorization \cite{schafer2020sparse, schaafer2021compression, chen2023sparse}. In some applications, these algorithms can be competitive (both in terms of complexity and accuracy) 
even when compared to highly optimized algebraic multigrid solvers such as AMGCL and Trilinos \cite{chen2021multiscale}. GP methods are naturally amenable to analysis and come with simple provable guarantees, while ANN-based methods involve complicated optimizations and many heuristics, which can make them hard to understand.
GP methods fit Occam's razor, offering a clarity of purpose in their structure. We can understand why and when they work, which is of scientific importance \cite{feynman1998cargo}. 
Therefore, it is increasingly important to benchmark deep learning methods against kernel-based methods to ensure that the deep part of a DL method serves a significant purpose beyond adding complexity.

\subsection{Outline of the Article}\label{sec:outline}
The rest of the article is organized as follows: We present a brief overview of our GP and kernel approach for 
solving nonlinear and parametric PDEs in \Cref{sec:review}; our error analysis is outlined in 
\Cref{sec:error-bounds}; followed by numerical experiments in \Cref{sec:numerics} and 
conclusions in \Cref{sec:conclusions}. Auxiliary results are collected in  \Cref{sec:AppA,sec:AppB,sec:AppC}.

\section{Kernel Methods for Parametric PDEs}\label{sec:review}

 In this section we extend the  kernel methodology of \cite{CHEN2021110668} 
 to the case of parametric PDEs as outlined in \Cref{subsec:param-PDE-solve-review}. 
 Some numerical strategies and ideas for improving the efficiency of the solver 
 are discussed  in \Cref{subsec:numerical-strategies}.

\subsection{Solving Parametric PDEs}\label{subsec:param-PDE-solve-review}
Let us consider bounded connected domains $\Omega \in \R^d$ with a Lipschitz boundary for $d \ge 1$
and $\Theta \subset \R^p$ for $p \ge 1$.
We  consider
 nonlinear  and parametric PDEs of the form
\begin{equation}
  \label{abstract-parametric-PDE}
  \left\{
    \begin{aligned}
      \mP(u^\star) (\bx; \btheta) & = f(\bx; \btheta), && (\bx, \btheta) \in
      \Omega \times \Theta, \\
      \mB(u^\star) (\bx; \btheta) & = g(\bx; \btheta), && (\bx, \btheta) \in
      \partial \Omega \times \Theta,
    \end{aligned}
    \right.
  \end{equation}
  where $\mP, \mB$ are nonlinear differential operators in the interior
  and boundary of $\Omega$,  $\btheta$ is a parameter,   
  and $f, g$ are the PDE source and boundary data. 
  For now
  we assume that the above PDE is well-posed and has
  a unique  solution $u^\star(\bx; \btheta)$ which is assumed to exist 
  in the strong sense over $\overline{\Omega}$ and for all values
  of $\btheta \in  \Theta$. 
  In \cite{CHEN2021110668} the authors introduced a GP/kernel method for 
  solving nonlinear PDEs of the form  \eqref{abstract-parametric-PDE} without 
  parametric dependence.
Here we  extend that approach to the  parametric case.

  Let $\Upsilon := \Omega \times \Theta$  and $\partial \Upsilon := \partial \Omega
\times \Theta$ and write $\bs := (\bx, \btheta)$. 
Choose $M \ge 1$ collocation points $\{ \bs_m \}_{m=1}^M \in
  \overline{\Upsilon}$ such that $\{ \bs_m \}_{m=1}^{M_\Omega}
  \in \Upsilon$ and $\{ \bs_m \}_{m=M_\Omega+ 1}^M
  \in \partial \Upsilon$\footnote{Note that we do not specifically ask for collocation points 
  on $\partial \Omega \times \partial \Theta$ since we may not have boundary data on the 
  $\btheta$ parameter.} and  consider a kernel
  $K: \overline{\Upsilon} \times \overline{\Upsilon}  \to \R$ with its corresponding RKHS denoted by $\mU$ and norm $\| \cdot \|_\mU$. 
  We then propose to approximate $u^\star(\bs)$ by solving the optimization
  problem
  \begin{equation}
    \label{parametric-PDE-kernel-opt}
    \left\{
      \begin{aligned}
        \minimize_{u \in \mU}  \quad & \| u \|_\mU \\
        \st \quad  & \mP(u)(\bs_m) = f(\bs_m), && m =1, \dots, M_\Omega, \\
            & \mB(u)(\bs_m) = g( \bs_m), && m=M_\Omega+1, \dots, M.
      \end{aligned}
      \right.
    \end{equation}
    Observe that our approach above approximates the solution $u$ as 
    a function defined on the set product set $\Upsilon$ which is different 
    from previous works \cite{beck2012optimal, Chkifa2012, chkifa2014high, nobile2008sparse, nobile2008anisotropic} 
    where the solution map $\btheta \to u^\star(\cdot; \btheta)$, as a mapping from $\Theta$ 
    to an appropriate function space,
    is characterized and approximated. The latter approach requires
    different discretization methods for the $\btheta$ parameter and
    the functions $u^\star(\cdot; \btheta)$ while our approach leads to
    a meshless collocation
    method on the product space which is desirable and convenient at the level of implementation, following \cite[Sec.~3.1]{CHEN2021110668} (see also \cite{giesl2007meshless}).
    
    We  make the following
    assumption on the differential operators $\mP, \mB$.
    \begin{assumption}\label{assumption:parametric-PDE}
      There exist bounded and linear operators
      $L_1, \dots, L_{Q_\Omega} \in  \mL \big(\mU; C(\Upsilon) \big)$  and 
      $L_{Q_\Omega + 1}, \dots, L_{Q}
      \in \mL\big(\mU; C(\partial \Upsilon) \big)$ for some
       $1 \le Q_\Omega < Q$
      together
      with maps $P: \R^{Q_\Omega} \to \R$ and $B: \R^{Q - Q_\Omega} \to \R$,
      which may
      be nonlinear, so that $\mP, \mB$ can be written as 
      \begin{equation}\label{Paramteric-PDE-differential-operator-assumption}
        \begin{aligned}
          \mP(u)( \bs) & = P \Big( L_1(u)( \bs), \dots, L_{Q_\Omega}(u) (\bs) \Big) && \forall \bs \in \Upsilon, \\
          \mB(u)( \bs) & = B \Big( L_{Q_\Omega + 1}(u)( \bs), \dots, L_{Q}(u) (\bs) \Big) && \forall \bs \in \partial \Upsilon.
      \end{aligned}
      \end{equation}
    \end{assumption}
    We briefly introduce a running example of a parametric PDE for which the above assumptions can 
    be verified easily. 

    \begin{example}[Nonlinear Darcy flow]\label{ex:darcy-flow}
      Consider the nonlinear Darcy flow PDE
      \begin{equation}\label{eq:darcy-flow-nonlinear}
        \left\{
        \begin{aligned}
          -\Div_{\bx} \big( \exp( a (\bx,\btheta) ) \nabla u \big) (\bx) + \tau(u (\bx) )  & = 1, && \bx \in \Omega, \\
          u(\bx) & = 0, && \bx \in  \partial \Omega,
        \end{aligned}
        \right.
    \end{equation}
    where $\Omega \subset \R^d$ is a bounded domain with a Lipschitz boundary and $\tau: \R \to \R$ is a continuous and
    nonlinear map. We assume that the
    permeability field $a$ is parameterized as
    \begin{equation}\label{coef-a-parameterization}
      a(\bx,\btheta) = \sum_{j=1}^p \theta_j \psi_j(\bx),
    \end{equation}
    where $\theta_j \in (0,1)$ so that $\Theta = (0,1)^p$ and $\psi_j \in C(\overline{\Omega})$.
    Substituting  into the PDE
    and expanding the differential operator we can rewrite our nonlinear PDE as
    \begin{equation*}
      \left\{
      \begin{aligned}
        -\exp \left( \sum_{j=1}^p \theta_j \psi_j(\bx) \right)
        \sum_{j=1}^p \theta_j  \nabla_\bx \psi_j(\bx) \cdot \nabla_\bx u(\bx; \btheta) \hspace{10ex} & &&
         \\  - \exp \left( \sum_{j=1}^p \theta_j \psi_j(\bx) \right)
        \Delta_\bx u(\bx; \btheta) + \tau(u(\bx; \btheta) ) &  = 1, && (\bx, \btheta) \in \Omega \times \Theta, \\
          u(\bx; \btheta) &  = 0, && (\bx, \btheta) \in  \partial \Omega \times \Theta,
        \end{aligned}
        \right.
      \end{equation*}
      where we used subscripts on the differential operators to highlight that
      derivatives are computed for the $\bx$ variable
      only and not $\btheta$. We also did not use the compact notation
      $\bs \equiv (\bx, \btheta)$ since it is more helpful to be able to
      distinguish between the $\bx$ and $\btheta$ variables in this example. 
      We can directly
      verify \Cref{assumption:parametric-PDE} with the bounded
      and linear operators
      \begin{equation*}
        \begin{aligned}
          L_1 & : u(\bx; \btheta) \mapsto u (\bx; \btheta), \\
          L_2 & : u(\bx; \btheta) \mapsto
          \exp \left( \sum_{j=1}^p \theta_j \psi_j(\bx) \right)
          \sum_{j=1}^p \theta_j  \nabla_\bx \psi_j(\bx) \cdot \nabla_\bx u(\bx; \btheta), \\
          L_3 &: u(\bx;\btheta) \mapsto
           \exp \left( \sum_{j=1}^p \theta_j \psi_j(\bx) \right)
           \Delta_\bx u(\bx; \btheta) \\
           L_4 & : u(\bx; \btheta) \mapsto u (\bx; \btheta),
      \end{aligned}
    \end{equation*}
    Note that operators $L_1, L_4$ are the same here since the point values of 
    $u$ appear in both the interior and boundary conditions.
    Thus we have $Q_\Omega = 3$ and $Q = 4$ and the maps
      \begin{equation*}
        P(t_1, t_2, t_3) = -t_2 - t_3 + \tau(t_1), \qquad B(t_1) = t_1.
      \end{equation*}
    \end{example}
    If $\mU$ is sufficiently regular and
    \Cref{assumption:parametric-PDE} holds, then we can
    define the functionals
    $\phi_m^{q} \in \mU^\ast$ for $1\leq q \leq Q$ as
    \begin{equation}
      \label{phi-def}
      \phi_{m}^{q} := \delta_{(\bs_m)} \circ L_q,
      \quad \text{where} \quad
      \left\{
      \begin{aligned}
        1 & \le m \le M_\Omega & \text{if} & & 1 \le q \le Q_\Omega \\
        M_\Omega+1 & \le m \le M & \text{if} & & Q_{\Omega+1} \le q \le Q. \\
      \end{aligned}
      \right.
    \end{equation}
    In what follows we write $[\phi, u]$ to denote the duality pairing between 
    $\mU$ and $\mU^\ast$ and 
     further use the shorthand notation $\bphi^{(q)}$ to denote the vector
    of dual elements $\phi^{q}_{m}$ for a fixed index $q$. Note that
    $\bphi^{(q)} \in ( \mU^\ast)^{\otimes M_\Omega}$ if $q \le Q_\Omega$
    but  $\bphi^{(q)} \in ( \mU^\ast)^{\otimes (M - M_\Omega)}$ if $q > Q_\Omega$
    in order to accommodate different differential
    operators defining the PDE and the boundary conditions.
    We further write $N = M_\Omega Q_\Omega + (M - M_\Omega) ( Q - Q_\Omega)$ and define
    \begin{equation*}
      \bphi = \big( \bphi^{(1)}, \dots, \bphi^{(Q)} \big) \in (\mU^\ast)^{\otimes N}. 
    \end{equation*}
    Henceforth we write $\phi_n$ for $n=1, \dots, N$ to denote the 
    entries of the vector $\bphi$ and write $[ \bphi, u ] = ( [ \phi_1, u], \dots, [\phi_N, u]) \in \R^N$.
    With this notation we rewrite problem \eqref{parametric-PDE-kernel-opt} as 
    \begin{equation*}
    \left\{
      \begin{aligned}
        \minimize  \quad & \| u \|_\mU \\
        \st \quad  & F( [ \bphi, u  ] ) = \by,
      \end{aligned}
      \right.
    \end{equation*}
    where the data vector $\by \in \R^M$ has entries
    \begin{equation*}
      y_m := \left\{
        \begin{aligned}
         & f(\bs_m),  && \text{if}  && 1 \le m \le M_\Omega, \\
         & g(\bs_m), && \text{if} && M_\Omega + 1 \le m \le M,
        \end{aligned}
        \right.
      \end{equation*}
    and  $F: \R^N \to \R^M$ is a nonlinear map whose output components are defined as
    \begin{equation}
        \label{F-definition}
        \big( F([ \bphi, u]) \big)_m := \left\{ 
        \begin{aligned}
         & P \left( [ \phi^{1}_{m}, u ], \dots, [\phi^{Q_\Omega}_{m}, u] \right) && \text{if } 1 \le m \le M_\Omega, \\
         & B \left( [ \phi^{Q_\Omega + 1}_{m}, u ], \dots, [\phi^{Q}_{m}, u]  \right) && \text{if } M_\Omega+ 1 \le m \le M.
        \end{aligned}
        \right.
    \end{equation}
    Further  define
    the kernel vector field
    \begin{equation}\label{def-kernel-vector-field}
      K(\cdot, \bphi): \overline{\Upsilon} \to \mU^N, \qquad
      K(\bs, \bphi)_k :=  [ \phi_k, K(\bs, \cdot) ]
    \end{equation}
    and the kernel matrix
    \begin{equation}\label{def-kernel-matrix}
      K(\bphi, \bphi) \in \R^{N \times N}, \qquad
      K(\bphi, \bphi)_{nk} :=
      [ \phi_n, K(\cdot, \bphi)_k ].
    \end{equation}
      We can then characterize the minimizers of \eqref{parametric-PDE-kernel-opt}
      via the following representer theorem which is a
      direct consequence of \cite[Prop.~2.3]{CHEN2021110668}:
      \begin{proposition}\label{representer-theorem-parametric-PDE}
        Suppose \Cref{assumption:parametric-PDE} holds and
         $K(\bphi, \bphi)$ is invertible. Then a function
        $u^\dagger: \overline{\Upsilon} \to \R$
        is a minimizer of \eqref{parametric-PDE-kernel-opt} if and only if
        \begin{equation}\label{representer-formula}
          u^\dagger(\bs) = K(\bs, \bphi) K(\bphi, \bphi)^{-1} \bz^\dagger,
        \end{equation}
        where $\bz^\dagger \in \R^N$ solves
        \begin{equation}\label{representer-opt-prob}
          \left\{
          \begin{aligned}
          \minimize_{\bz \in \R^N} \quad  & \bz^T K(\bphi, \bphi)^{-1} \bz, \\
           \st \quad & F(\bz) = \by,
         \end{aligned}
         \right.
        \end{equation}
        with the nonlinear map $F$ defined in \eqref{F-definition}.
      \end{proposition}
      
This result allows us to reduce the infinite-dimensional optimization problem 
\eqref{parametric-PDE-kernel-opt} to a finite-dimensional optimization problem without incurring 
any approximation errors; it is an instance of the well-known family of representer theorems \cite[Sec.~4.2]{scholkopf2018learning}. Thus, to find an approximation to $u^\star$ we simply need to solve \eqref{representer-opt-prob} and apply the formula \eqref{representer-formula}; algorithms 
for this task are discussed next.

\subsection{Numerical Strategies}\label{subsec:numerical-strategies}
We now summarize various numerical strategies for solution of \eqref{representer-opt-prob}. These strategies 
are naturally applicable to non-parametric PDEs  as they can 
be viewed as a special case of \eqref{abstract-parametric-PDE} with a fixed parameter. 
In \Cref{subsec:Gauss-Newton} we summarize a Gauss-Newton algorithm 
that was introduced in \cite{CHEN2021110668} followed by a new and, often, more 
efficient strategy that linearizes the PDE first before formulating 
the optimization problem in \Cref{subsec:linearize-then-optimize}. 

\subsubsection{Gauss-Newton}\label{subsec:Gauss-Newton}
To solve the optimization problem \eqref{representer-opt-prob}, a Gauss-Newton algorithm was proposed in \cite{CHEN2021110668} which we recall briefly. 
The equality constraints can be dealt with either by elimination or relaxation.
Suppose that there exists a map $\hF: \R^{N - M} \times \R^M \to \R^N$ so that
 \begin{equation*}
   F (\bz) = \by \quad  \text{ if and only if } \quad
   \bz = \hF(\bw, \by), \quad \text{ for a unique } \bw \in \R^{N - M}\, .
\end{equation*}
Then, we rewrite \eqref{representer-opt-prob} as the unconstrained optimization 
problem
\begin{equation}\label{PDE-unconstrained-representer-theorem}
  \minimize_{\bw \in \R^{N - M}} \quad  \hF(\bw, \by)^T K(\bphi, \bphi)^{-1} \hF(\bw, \by)\, .
\end{equation}
Then  a minimizer $\bw^\dagger$ of  \eqref{PDE-unconstrained-representer-theorem}
can be approximated with a sequence of elements 
$\bw^{\ell}$ defined iteratively via
  $\bw^{\ell + 1} = \bw^{\ell} + \alpha^\ell  \delta \bw^{\ell},$
where $\alpha^\ell > 0 $ is an appropriate step size while
$\delta \bw^{\ell}$ is the minimizer of the optimization problem
\begin{equation*}\label{GN-alg-optimization-unconstrained}
  \minimize_{ \delta \bw \in \R^{N - M}} \quad
  \left( \hF(\bw^\ell, \by) +    \nabla_\bw  \hF(\bw^\ell, \by) \delta\bw \right) ^T
  K(\bphi, \bphi)^{-1} \left( \hF(\bw^\ell, \by) 
  +    \nabla_\bw  \hF(\bw^\ell, \by) \delta\bw \right)\,.
\end{equation*}

Alternatively, if the map $\hF$ does not exist or is hard to compute, i.e.,
eliminating the constraints is not feasible, then we consider the relaxed 
problem 
\begin{equation*}
    \minimize_{\bz \in \R^N} \quad
    \frac{1}{2}\bz^T K(\bphi, \bphi)^{-1} \bz + \frac{1}{2\beta^2} | F(\bz) - \by |^2, 
\end{equation*}
for a sufficiently small parameter $\beta>0$. Here $\|\cdot\|$ is the $L^2$ norm of the vector. A minimizer $\bz^\dagger_\beta$
of the above problem can be approximated with a sequence $\bz^\ell$ 
where $\bz^{\ell+1} = \bz^\ell + \alpha^\ell \delta \bz^\ell$ 
where $\delta \bz^\ell$ is the minimizer of 
\begin{equation*}
    \minimize_{\delta \bz \in \R^N} \quad
    \delta \bz^T K(\bphi, \bphi)^{-1}  \bz^\ell 
    + \frac{1}{2\beta^2} | F(\bz^\ell) +  \nabla F(\bz^\ell) \delta \bz - \by |^2.
\end{equation*}

\subsubsection{Linearize then Optimize}\label{subsec:linearize-then-optimize}
The Gauss-Newton approach of \Cref{subsec:Gauss-Newton} is 
applicable to  wide families of nonlinear PDEs. The primary 
computational bottleneck of that approach is the construction and 
factorization of the kernel matrix $K(\bphi, \bphi)$ which for some PDEs can be prohibitively large. To get around this 
difficulty we propose an alternative approach to approximating 
the solution of \eqref{parametric-PDE-kernel-opt} by first linearizing the 
PDE operators before applying \Cref{representer-theorem-parametric-PDE}. The resulting 
approach is more intrusive in comparison to the Gauss-Newton method 
as it requires explicit calculations involving the PDE but often 
leads to smaller kernel matrices and better performance. This method can 
also be viewed as applying the methodology of \cite{CHEN2021110668,giesl2007meshless} to discretize 
successive linearizations of the PDE. 

Let $u^\dagger$ denote the minimizer of \eqref{parametric-PDE-kernel-opt} 
as before. Assuming that the operators $\mP$ and $\mB$ are 
Fr\'echet differentiable we then approximate $u^\dagger$ with a sequence of elements 
$u^\ell$ obtained by solving the problem 
 \begin{equation}\label{linearized-parametric-PDE-opt}
    \left\{
      \begin{aligned}
        \minimize_{u \in \mU} \quad & \| u \|_\mU \\
        \st \quad  & \left(\mP(u^{\ell-1}) + \mP'(u^{\ell-1})(u-u^{\ell-1})\right)|_{\bs_m} = f(\bs_m), && m =1, \dots, M_\Omega, \\
            & \left( \mB(u^{\ell-1}) + \mB'(u^{\ell-1})(u-u^{\ell-1})\right)|_{\bs_m} = g( \bs_m), && m=M_\Omega+1, \dots, M,
      \end{aligned}
      \right.
    \end{equation}
where, $\mP'$ and $\mB'$ are the Fr\'echet derivatives of 
$\mP$ and $\mB$. 

Let us further suppose that \Cref{assumption:parametric-PDE}
holds. Observing that 
the constraints in \eqref{linearized-parametric-PDE-opt} are linear in $u$
we obtain an explicit formula for $u^\ell$ by
\cite[Prop.~2.2]{CHEN2021110668}:
\begin{equation}
u^{\ell}(\bs) = K(\bs,\tilde{\bphi}^{\ell-1})K(\tilde{\bphi}^{\ell-1},\tilde{\bphi}^{\ell -1})^{-1} \bz^{\ell-1}
\end{equation}
where $\bz^{\ell-1} = ( z_1^{\ell -1}, \dots, z^{\ell-1}_M)^T$ has entries
\begin{equation*}
 z^{\ell-1}_m = \left\{ 
    \begin{aligned}
    &\left(f-\mP(u^{\ell-1})+\mP'(u^{\ell-1})u^{\ell-1}\right)|_{\bs_m}, 
    && \text{if } 1 \le m \le M_\Omega, \\ 
        &\left(f-\mP(u^{\ell-1})+\mB'(u^{\ell-1})u^{\ell-1}\right)|_{\bs_m}, 
    && \text{if } M_\Omega+1 \le m \le M.
    \end{aligned}
 \right.
\end{equation*}
The vectors $\tilde{\bphi}^{\ell-1} \in ( \mU^\ast) ^{\otimes M}$ are obtained by concatenating the dual elements 
\begin{equation*}
    \tilde{\phi}_{m}^{\ell-1} := \left\{
    \begin{aligned}
    & \delta_{(\bs_m)} \circ \mP'(u^{\ell-1}), && \text{if } m = 1, \dots, M_\Omega, \\ 
    & \delta_{(\bs_m)} \circ \mB'(u^{\ell-1}), && \text{if } m = M_\Omega+1, \dots, M. \\ 
    \end{aligned}
    \right.
\end{equation*}
We note that the above scheme implicitly assumes that $\mU$ is sufficiently regular so that the derivatives
$\mP'(u^{\ell-1})$ and $\mB'(u^{\ell-1})$ can be  regarded as linear operators  mapping $\mU$ to $C(\overline{\Upsilon})$
where pointwise evaluation is well-defined. The most important feature of the linearize-then-optimize approach 
is that the kernel matrices $K(\tilde{\bphi}^{\ell-1},\tilde{\bphi}^{\ell -1})$ are of size $M \times M$
while the kernel matrix $K(\bphi, \bphi)$, used in the Gauss-Newton approach of \Cref{subsec:Gauss-Newton}, is 
of size $N \times N$. Thus, the linearize-then-optimize approach requires the inversion of much smaller kernel matrices at 
each iteration but these matrices need to be updated successively since the $\tilde{\bphi}^\ell$ depend on the 
 previous solution $u^{\ell -1}$.
In the case where \Cref{assumption:parametric-PDE} holds and $P, B$ are differentiable, then the $\mP'(u)$ and $\mB'(u)$ operators 
can be written explicitly as 
\begin{equation*}
\begin{aligned}
    \mP'(u^{\ell-1}): \: & u \mapsto \nabla P \left( L_1(u^{\ell-1}) , \dots, L_{Q_\Omega}(u^{\ell-1}) \right)^T 
    \begin{bmatrix}
    L_1(u) \\ 
    \vdots \\
    L_{Q_\Omega}(u)
    \end{bmatrix},  \\
        \mB'(u^{\ell-1}): \: & u \mapsto \nabla B \left( L_{Q_\Omega +1}(u^{\ell-1}) , \dots, L_{Q}(u^{\ell-1}) \right)^T 
    \begin{bmatrix}
    L_{Q_\Omega + 1}(u)\\ 
    \vdots \\
    L_{Q}(u)
    \end{bmatrix}.
\end{aligned}
\end{equation*}
\begin{remark}
    In \cite[Sec. 5.1]{chen2023sparse}, it is shown that the Gauss-Newton iteration and the ``linearize then optimize'' approach are mathematically equivalent for some PDEs. Specifically, they are both equivalent to a sequential quadratic programming approach to solving \eqref{representer-opt-prob}.
\end{remark}

\section{Error Analyses}\label{sec:error-bounds}
We now present our main theoretical results concerning convergence rates for the 
minimizers $u^\dagger$ of  \eqref{parametric-PDE-kernel-opt} to the respective true solutions $u^\star$. 
We start in \Cref{subsec:abstract-error-analysis} by articulating
the abstract framework, main theorem and proof. We then consider the simple setting of a nonlinear PDE 
in \Cref{subsec:PDE-error-bounds} where the RKHS 
$\mU$ already satisfies the boundary conditions of the PDE to convey the main ideas of the 
proof in a simple setting.  
Non-trivial boundary conditions are then considered in \Cref{subsec:PDE-error-bounds-with-BC}
followed by the case of parametric PDEs  in \Cref{subsec:Parametric-PDE-error-bounds}.
Our proof technique is a generalization of the results of \cite{giesl2007meshless, schaback2016all} to the case of nonlinear and parametric PDEs that are Lipschitz stable and well-posed.

\subsection{An Abstract Framework for Obtaining Convergence Rates}\label{subsec:abstract-error-analysis}
We present here an abstract theoretical result that allows us to obtain convergence rates for 
nonlinear operator equations. Our error analyses concerning the  numerical solutions $u^\dagger$
and the true solution to the PDE $u^\star$ then follow as applications of this abstract result. 
Our main result here can also be viewed as a generalization of the results of 
\cite[Sec. 10]{schaback2016all}, which focused on linear operators, to the nonlinear case. 

Let us consider operator equations of the form 
\begin{equation}\label{abstract-nonlinear-operator-equation}
    \mT(v^\star) = w^\star,
\end{equation}
where $v^\star, w^\star$ are elements of  appropriate Banach spaces and $\mT$ is 
a nonlinear map. 
In the setting of PDEs the map $\mT$ is defined by the differential operator of the PDE, 
$v^\star$ coincides with the solution and $w^\star$ is the source/boundary data. 
Broadly speaking our goal is to approximate the solution $v^\star$ under assumptions on 
its regularity and the stability properties of the map $\mT$.
To this end, we present a general result that allows us to control the error of 
approximating $v^\star$ given an appropriate candidate $v^\dagger$.

\begin{theorem}\label{thm:abstract-error-analysis}
Assume problem \eqref{abstract-nonlinear-operator-equation} is uniquely solvable  
and consider abstract Banach spaces $(V_i, \| \cdot \|_i)_{i=1}^4$ as well as 
$(\mU, \| \cdot  \|_\mU)$. 
Let $v^\star, v^\dagger \in \mU$ and suppose 
the following conditions are satisfied for any choice of $r >0$ (all the appeared constants $C(r)$ are non-decreasing regarding $r$): 
\begin{enumerate}[label=(A\arabic*)]

\item For any pair $v, v' \in B_r(V_1)$ 
there exists a constant $C = C(r) > 0$ so that
    \begin{align}
            \| v - v' \|_1 & \le C \| \mT(v) - \mT(v') \|_2.
            \label{abs-inverse-stability} 
    \end{align}
    \label{cond:A1}

\item There exists a constant $\veps>0$, independent of $v^\dagger$ and $v^\star$ so that 
    \begin{align}
        \| \mT(v^\dagger) - \mT(v^\star) \|_2 \le \veps \| \mT(v^\dagger) - \mT(v^\star) \|_3.  
            \label{abs-cond-2} 
    \end{align}
    \label{cond:A2}
\item For any pair $v, v' \in B_r(V_4)$ there exists a constant $C = C(r) >0$ so that 
\begin{align}
            \| \mT(v) - \mT(v') \|_3 & \le C \| v - v' \|_4.
            \label{abs-forward-stability} 
    \end{align}
    \label{cond:A3}

 \item For any $v \in V_4$,  there exists a constant $C>0$ so that 
     \begin{equation*}
     \|v\|_4\leq C \|v\|_\mU.    
     \end{equation*}
     \label{cond:A4}

\item There exists a constant $C > 0$, independent of $v^\star$ and $v^\dagger$ so that 
\begin{align}
\| v^\dagger \|_\mU & \le C \| v^\star \|_\mU.     \label{abs-cond-1}    
\end{align}
\label{cond:A5}

\end{enumerate}
Then there exists a constant  $C$ that depends on $\| v^\star \|_\mU$ such that 
\begin{equation*}
    \| v^\dagger - v^\star \|_1 \le C \veps \| v^\star \|_\mU. 
\end{equation*}
\end{theorem}

\begin{proof}
The conditions of the theorem are presented in the order of the logical steps of the proof.
By \ref{cond:A1} we have that 
\begin{equation}\label{abstract-proof-disp}
    \| v^\dagger - v^\star \|_1 \le C \| \mT(v^\dagger) - \mT(v^\star) \|_2.
\end{equation}
Then 
 \ref{cond:A2}  implies that  
    $\| \mT(v^\dagger) - \mT(v^\star) \|_2 \le  C \veps \| \mT(v^\dagger) - \mT(v^\star) \|_3$. By the triangle inequality 
    we  have $ \| \mT(v^\dagger) - \mT(v^\star)  \|_3 \le  \| \mT(v^\dagger) - \mT(0) \|_3 
    + \| \mT(v^\star) - \mT(0) \|_3$.
    Using \ref{cond:A3}, \ref{cond:A4}, 
    and \ref{cond:A5} in that order, we get $\| \mT(v^\dagger) - \mT(0) \|_3
        \le C \| v^\dagger \|_4  
        \le C \| v^\dagger \|_\mU
        \le C \| v^\star \|_\mU.$
     Similarly, we have $\| \mT(v^\star) - \mT(0) \|_3 \le C \| v^\star \|_\mU$. 
    Combining these bounds we obtain  $ \| \mT(v^\dagger) - \mT(v^\star) \|_3 \le C \veps \| v^\star \|_\mU$ which yields the desired result due to \eqref{abstract-proof-disp}.
    {}
\end{proof}


Let us provide some remarks regarding the assumptions of the theorem. In our PDE examples 
we often take the $V_i$ spaces to be Sobolev spaces of appropriate smoothness while $\mU$ 
is taken as an RKHS that is sufficiently smooth and so  $v^\star \in \mU$ 
amounts to an assumption on the regularity of the true solution to the problem. 
Conditions \ref{cond:A1}
and \ref{cond:A3} amount to forward and inverse Lipschitz stability of the operator $\mT$
while \ref{cond:A2} is often given by a sampling/Poincar{\'e}-type inequality for our numerical 
method. 
We treat the constant $\veps$ separately from the other constants in the theorem since in practice 
$\veps$ often coincides with some power of the resolution (fill-distance/meshnorm) of our numerical scheme, 
constituting the rate of convergence of the method. 
Assumption \ref{cond:A4} also concerns the regularity of the RKHS and the choice of 
the space $V_4$ (we simply ask for $\mU$ to be continuously embedded in $V_4$) and is a 
matter of the setup of the problem. Condition \ref{cond:A5} is less natural as it 
requires the norm of the approximate solution $v^\dagger$ to be controlled by the norm of $v^\star$. 
While this condition does not hold for many numerical approximation schemes, we will see that it 
follows easily from the setup of our collocation/optimal recovery scheme.

  In plain words, the most important message of \Cref{thm:abstract-error-analysis} is that: 
  {\it given Condition \eqref{abs-cond-1} and the Lipschitz-continuity of $\mT$ and its inverse, it 
  follows that the approximation error between $v^\dagger$ and $v^\star$ is bounded by the 
  approximation error between  $\mT(v^\dagger)$ and $\mT(v^\star)$}. This result can be applied to 
 both GP/kernel  and ANN based collocation methods, since both seek to minimize the error between 
 $\mT(v^\dagger)$ and $\mT(v^\star)$ at collocation points. This Condition 
 \eqref{abs-cond-1} is automatically satisfied for our GP/Kernel based methods that solve problems of 
 the form 
\begin{equation*}
    v^\dagger = \argmin \| v \|_\mU 
    \qquad \text{s.t.} \qquad  
    [\phi_i, \mT(v)]  = [ \phi_i, \mT(v^\star)], \quad i =1, \dots, M,
\end{equation*}
with $\mT$ denoting the differential operator of a PDE and $\phi_i$ denoting a 
set of dual elements (e.g. pointwise evaluations at collocation points). Then since the true solution $v^\star$ satisfies  the PDE for an infinite collection of dual elements (e.g. pointwise within a set, or in a weak sense) then we immediately have that $\| v^\dagger \|_\mU \le \| v^\star \|_\mU$. One can also take $\mU$ to be a Barron space
(indeed the $V_i$ norms could be arbitrary)
to obtain an analogous result 
for ANNs, but it is unclear if this setup coincides with (or leads to) any practical algorithms.

\subsection{The Case of Second Order Nonlinear PDEs}\label{subsec:PDE-error-bounds}
We begin our error analysis in the case where \eqref{abstract-parametric-PDE} 
does not depend on the parameter $\btheta$ and homogeneous Dirichlet boundary 
conditions are imposed, i.e., nonlinear second order  PDEs of the form 
\begin{equation}
  \label{abstract-PDE}
  \left\{
    \begin{aligned}
      \mP(u^\star) ( \bx) & = f( \bx), && \bx \in \Omega,\\
      u^\star (\bx) & = 0, && \bx \in \partial\Omega.
    \end{aligned}
    \right.
  \end{equation}
  The choice of Dirichlet boundary conditions is only made for simplicity here and 
  can be replaced with other conditions of interest. We will also consider approximate 
  boundary conditions in \Cref{subsec:PDE-error-bounds-with-BC}. 
 We  further assume that 
 the kernel $K$ is chosen so that the elements of $\mU$ 
readily satisfy the boundary conditions of the PDE and consider optimization problems of the form 
  \begin{equation}
    \label{PDE-kernel-opt}
    \left\{
      \begin{aligned}
        \minimize_{u \in \mU}  \quad & \| u \|_\mU \\
        \st \quad  & \mP(u)(\bx_m) = f(\bx_m), && m =1, \dots, M_\Omega, \\ 
        & u(\bx) = 0, && \bx \in \partial \Omega.
      \end{aligned}
      \right.
    \end{equation}
    where $ X_\Omega := \{ \bx_m \}_{m=1}^{M_\Omega} \subset \Omega$ are a set of collocation points. 
We need  to impose appropriate assumptions on the RKHS $\mU$,  the domain $\Omega$ and the PDE operator $\mP$.

\begin{assumption}\label{assumption:PDE-regularity}
The following conditions hold: 
\begin{enumerate}[label=(B\arabic*)]
    \item {\it (Regularity of the domain)} $\Omega \subset \R^d$ is a compact set with a Lipschitz boundary.
    \label{assumption-PB-i}
    \item 
    {\it (Stability of $\mP$)}   
    There exist indices  $\gamma > 0$ and $k \in \mathbb{N}$ satisfying 
    $ d/2 < k + \gamma$ and $s\geq 1, \ell \in \R$, so that  for any $r > 0$
    it holds that
    \begin{align}
        \| u_1 - u_2 \|_{H^{\ell}(\Omega)}  &\le C \| \mP(u_1) - \mP(u_2) \|_{H^{k}(\Omega)}, 
        && \forall u_1, u_2 \in B_r(H^\ell(\Omega) \cap H^1_0(\Omega))
        \label{eq:PDE-stability-estimate-backward} \\
        \| \mP(u_1) - \mP(u_2) \|_{H^{k+ \gamma}(\Omega)} &\le C \| u_1 - u_2 \|_{H^{s}(\Omega)}, 
        && \forall u_1, u_2 \in B_r(H^s(\Omega) \cap H^1_0(\Omega))
        \label{eq:PDE-stability-estimate-forward}  
    \end{align}
    where 
    $C = C(r) >0$ is independent of the $u_i$'s. The space  $H^s(\Omega) \cap H^1_0(\Omega)$ can be equipped with the norm $\|\cdot\|_{H^s(\Omega)}$, which is used to define the balls above. \label{assumption-PB-ii}
    \item $\mU$ is continuously embedded in $H^s(\Omega) \cap H^1_0(\Omega)$.
    \label{assumption-PB-iii}
\end{enumerate}
\end{assumption}

\Cref{assumption-PB-i} is standard while 
\ref{assumption-PB-iii} dictates the choice of the RKHS $\mU$, and in turn the kernel, 
 which should be made based on a priori knowledge about regularity of the 
strong solution $u^\star$. We highlight that, 
asking elements of $\mU$ to satisfy the boundary conditions 
is only practical for simple domains and boundary conditions such as periodic, Dirichlet, or 
Neumann conditions  on  hypercubes or spheres.
Assumption \ref{assumption-PB-ii}  on the other hand is a question in the analysis of nonlinear PDEs
and is independent of our numerical scheme; simply put we require the PDE to be 
Lipschitz well-posed with respect to the right hand side/source term.

We are now ready to present our first theoretical result characterizing the convergence of 
the minimizer $u^\dagger$ of \eqref{PDE-kernel-opt} 
to $u^\star$ the strong solution of \eqref{abstract-PDE}. 

\begin{theorem}\label{thm:PDE-error-bound-no-BC}
Suppose \Cref{assumption:PDE-regularity}
is satisfied and let $u^\star \in \mU$ denote the unique strong solution of \eqref{abstract-PDE}.
Let $u^\dagger$ be a minimizer of \eqref{PDE-kernel-opt} 
with a set of collocation points $X_\Omega \subset \Omega$
and define their fill-distance
 $$   h_\Omega : =  \sup_{\bx' \in \Omega} \inf_{\bx \in X_\Omega} | \bx' - \bx |.$$
Then there exists a constant $h_0 >0$ so that if $h_\Omega < h_0$ then 
\begin{equation*}
    \| u^\dagger - u^\star \|_{H^{\ell}(\Omega)} \le C h_\Omega^{\gamma} \| u^\star \|_\mU,
\end{equation*}
where the constant $C> 0$ is independent of $u^\dagger$, and $h_\Omega$.
\end{theorem}

\begin{proof}
We will obtain the result by applying \Cref{thm:abstract-error-analysis} with 
 the map $\mT \equiv \mP$ and  the spaces 
$V_1 \equiv H^\ell(\Omega)$, $V_2 \equiv H^k(\Omega)$, $V_3 \equiv H^{k+\gamma}(\Omega)$, 
and $V_4 \equiv H^s(\Omega)$. 
With this setup we proceed to verify the conditions of \Cref{thm:abstract-error-analysis}: 
Condition \ref{cond:A1} follows from \eqref{eq:PDE-stability-estimate-backward}, 
\ref{cond:A3} follows from \eqref{eq:PDE-stability-estimate-forward}, 
\ref{cond:A4} follows from \ref{assumption-PB-iii}. 

Condition \ref{cond:A5} 
holds since $u^\dagger$ is a minimizer 
of \eqref{PDE-kernel-opt} and so  $\| u^\dagger \|_\mU \le \| u^\star \|_\mU$, 
since $u^\star$ is feasible but satisfies additional constraints compared with $u^\dagger$, 
i.e., it solves the PDE over the entire set $\Omega$. Thus, \eqref{abs-cond-1} is verified
with constant $C=1$. 

It remains to verify \ref{cond:A2}:
Let $\bar{f} = \mP(u^\dagger) - \mP(u^\star)$  and observe that $\bar{f}(\bx) = 0$ for all $\bx \in X_\Omega$.  Thus $\bar{f} \in H^{k+\gamma}(\Omega)$ is zero on $X_\Omega$ and an application of \Cref{fuselier-sobolev-bound} yields 
the existence of a constant $h_0>0$ so that whenever $h_\Omega < h_0$ then
$ \| \bar{f} \|_{H^k(\Omega)} \le C h_\Omega^{\gamma} \| \bar{f} \|_{H^{k + \gamma}(\Omega)}. $
This verifies \eqref{abs-cond-2} with $\veps \equiv C h_\Omega^\gamma$.
\end{proof}

\begin{remark}\label{rem:local-stability}
We note that \Cref{assumption-PB-ii} and in turn \Cref{thm:PDE-error-bound-no-BC} can easily be modified to a local version where 
the stability estimates \eqref{eq:PDE-stability-estimate-forward} and \eqref{eq:PDE-stability-estimate-forward}
are stated for $u_1, u_2$ belonging 
to a ball of radius $r>0$ around the true solution $u^\star$. Then one can obtain 
an asymptotic rate for $\| u^\dagger -u^\star \|_{H^\ell(\Omega)}$ under the additional 
assumption that $u^\dagger$ is sufficiently close to $u^\star$.
\end{remark}

\begin{remark}
The assumptions and results of \Cref{assumption:PDE-regularity} are analogous to the one used to obtain error estimates in numerical homogenization for elliptic PDEs \cite{owhadi2019operator}. In particular \Cref{thm:PDE-error-bound-no-BC} can be extended to the setting where measurements on the PDE are not pointwise but involve integral operators and where the coefficients may be rough. 
\end{remark}

We now present  a brief example where 
\Cref{assumption:PDE-regularity} can be verified and so \Cref{thm:PDE-error-bound-no-BC} is applicable to obtain convergence rates for our GP/kernel collocation solver.

\begin{example}[Nonlinear Darcy flow continued]\label{ex:Darcy-2}
Let us consider the nonlinear Darcy flow PDE \eqref{eq:darcy-flow-nonlinear} and  assume 
that $\Omega$ has a smooth boundary and $a(\bx) \in C^\infty(\Omega)$ is fixed and  satisfies $a(\bx) \ge 1$.
Further suppose $\tau(z) = 1 + \tanh(\beta z)$ for a fixed constant $\beta >0$ 
to be determined. Now pick $k = \lceil d/2 + \alpha \rceil$ from which 
it follows that
$H^k(\Omega)$ is continuously embedded in $C^\alpha(\bar{\Omega})$ \cite[Thm.~7.26]{gilbarg1977elliptic}
and fix an integer $\gamma > 0$\footnote{We only assume the exponents are integers 
for simplicity but our arguments can be generalized to the case of non-integer indices}. It is then straightforward to verify \eqref{eq:PDE-stability-estimate-forward} with $s = k + \gamma + 2$
as well as \Cref{assumption:PDE-regularity}(iii) by choosing  
the kernel $K$ to be the Green's function 
of the operator $(-\Delta)^{s}$ on the domain $\Omega$, subject to homogeneous Dirichlet boundary 
conditions. We note that approximating this Green's function can be expensive in practice and in \Cref{subsec:PDE-error-bounds-with-BC} we propose a way around this step by using collocation points on the boundary of $\Omega$ to impose the boundary 
conditions. 

Furthermore our assumptions on $a$ and $\tau$ imply that  $\mP$ is uniformly elliptic
in $\bar{\Omega}$ (see \cite[Part~II]{gilbarg1977elliptic} 
for definition of ellipticity for nonlinear elliptic PDEs).
Since $\tau$ is smooth it follows from \cite[Thm.~13.8]{gilbarg1977elliptic} that, 
for any $\alpha \in (0,1)$ and $f \in C^\alpha(\bar{\Omega})$, the PDE 
     \begin{equation}\label{Darcy-PDE-example}
        \left\{
        \begin{aligned}
          -\Div \big( \exp(a ) \nabla u \big) + \tau(u )  & = f, && \bx \in \Omega, \\
          u & = 0, && \bx \in  \partial \Omega,
        \end{aligned}
        \right.
    \end{equation}
has a solution $u \in C^{2}(\bar{\Omega})$. Now pick $f_1, f_2 \in H^k(\Omega)$ 
which, by the aforementioned Sobolev embedding result, belong to $C^\alpha(\bar{\Omega})$.
Write $u_1, u_2 \in C^2(\bar{\Omega})$ for the solution of the PDE with both right hand sides and observe that 
the difference $w : = u_1 - u_2$ solves the PDE 
\begin{equation*}
    -\Div \big( \exp(a  ) \nabla w \big) = f_1 - f_2 + \tau(u_2) - \tau(u_1).
\end{equation*}
Standard stability results for linear elliptic PDEs then imply the bound
\begin{equation*}
     \| w\|_{H^2(\Omega)}  \le B \left( \| f_1 - f_2 \|_{L^2(\Omega)} + \| \tau(u_2) - \tau(u_1) \|_{L^2(\Omega)} \right), 
\end{equation*}
for a constant $B> 0$ independent of $w, f_1, f_2, u_1, u_2$. Since $\tau$ is globally $\beta$-Lipschitz
we infer that $\| \tau(u_1) - \tau(u_2) \|_{L^2(\Omega)} \le \beta \| w\|_{L^2(\Omega)}$ which, 
together with the subsequent bound, yields 
    \begin{equation*}
    \| w\|_{H^2(\Omega)}\le  \frac{B}{1-B\beta}  \| f_1 - f_2 \|_{L^2(\Omega)}.
     \end{equation*}
Thus, assumption \eqref{eq:PDE-stability-estimate-backward} is satisfied 
with $\ell = 2$ as long as $\beta B < 1$. 

\end{example}

\subsection{ Handling Boundary Conditions}\label{subsec:PDE-error-bounds-with-BC}
We now turn our attention to the case where \eqref{abstract-parametric-PDE} is still independent of 
the $\btheta$ parameter but involves non-trivial boundary conditions, i.e., 
\begin{equation}
  \label{abstract-PDE-with-BC}
  \left\{
    \begin{aligned}
      \mP(u^\star) ( \bx) & = f( \bx), && \bx \in \Omega,\\
      \mB(u^\star) (\bx) & = g(\bx), && \bx \in \partial\Omega.
    \end{aligned}
    \right.
  \end{equation}
We will further assume that the elements of $\mU$ do not satisfy the boundary conditions exactly and 
so boundary collocation points are utilized to approximate those conditions leading to the problem 
 \begin{equation}
    \label{PDE-kernel-opt-with-BC}
    \left\{
      \begin{aligned}
        \minimize_{u \in \mU}  \quad & \| u \|_\mU \\
        \st \quad  & \mP(u)(\bx_m) = f(\bx_m), && m =1, \dots, M_\Omega, \\ 
        & \mB(u)(\bx_m) = g(\bx_m), && m = M_\Omega+1, \dots, M.
      \end{aligned}
      \right.
    \end{equation}
    where $ X_\Omega := \{ \bx_m \}_{m=1}^{M_\Omega} \subset \Omega$ are the interior  collocation points as before 
    and $ X_{\partial \Omega} := \{ \bx_m \}_{m=M_\Omega+1}^M \subset \partial \Omega$ are the boundary collocation points.
We will state our assumptions and results for PDEs in $d>1$ dimensions since in the 
1D case we can, in principle, impose the boundary conditions exactly by placing some collocation points on boundary.  
The main difference, in comparison to 
 \Cref{thm:PDE-error-bound-no-BC}, is that here we need to impose 
new assumptions on the PDE operators $\mP$ and $\mB$ and the boundary of $\Omega$  to be able to 
use  \Cref{fuselier-sobolev-bound} (sampling inequality on manifolds) in the final step of the proof
to obtain approximation rates for the boundary data.

\begin{assumption}\label{assumption:PDE-regularity-with-BC}
The following conditions hold: 
\begin{enumerate}[label=(C\arabic*)]
    \item {\it (Regularity of the domain and its boundary)} $\Omega \subset \R^d$ with $d >1$
    is a compact set and $\partial \Omega$ is a smooth connected Riemannian manifold 
    of dimension $d-1$ endowed with a 
    geodesic distance $\rho_{\partial \Omega}$. 
    \label{assumption-PC-i}

    \item {\it (Stability of the PDE)}
    
    There exist $\gamma>0$ and 
    $k, t \in \mathbb{N}$ satisfying $d/2 < k+\gamma$ and $(d-1)/2 < t+\gamma$, and $s,\ell \in \R$,
    so that for any $r > 0$  it holds that
    \label{assumption-PC-ii}
        \begin{equation}\label{eq:PDE-stability-estimate-backward-with-BC}
    \begin{aligned}
    \| u_1 - u_2 \|_{H^\ell(\Omega)} \le
     & C  \big( \| \mP(u_1) - \mP(u_2) \|_{H^k(\Omega)} \\ 
     & + \| \mB(u_1) - \mB(u_2) \|_{H^t(\partial \Omega)}  \big) \quad  \forall u_1, u_2 \in B_r(H^\ell(\Omega)),
    \end{aligned}
    \end{equation}
    \begin{equation}\label{eq:PDE-stability-estimate-forward-with-BC}
    \begin{aligned}
    \| \mP(u_1) - \mP(u_2) \|_{H^{k+\gamma}(\Omega)} &
    + \| \mB(u_1) - \mB(u_2) \|_{H^{t + \gamma}(\partial \Omega)} \\ 
    & \le C \| u_1 - u_2 \|_{H^{s}(\Omega)}, \qquad \forall u_1, u_2 \in B_r(H^s(\Omega)),
    \end{aligned}
    \end{equation}
    where $C = C(r)>0$ is a constant independent of the $u_i$. 
    
    \item $\mU$ is continuously 
    embedded in $H^s(\Omega)$. \label{assumption-PC-iii}
    \end{enumerate}
\end{assumption}

Observe that the above assumptions are analogous to \Cref{assumption:PDE-regularity} 
with the exception that we no longer work with the restricted Sobolev spaces $H^k_0$ since 
we do not need to impose the boundary conditions. However, we need to state our stability results 
for both $\mP$ and $\mB$. We emphasize that the verification of condition \ref{assumption-PC-ii} remains a
question in the analysis of PDEs.
We are now ready to extend Theorem~\ref{thm:PDE-error-bound-no-BC}
to the case of non-trivial boundary conditions. 

\begin{theorem}\label{thm:PDE-error-bound-with-BC}
Suppose \Cref{assumption:PDE-regularity-with-BC} is satisfied and let 
 $u^\star \in \mU$ denote the unique strong  solution of \eqref{abstract-PDE}.
Let $u^\dagger$ be a minimizer of \eqref{PDE-kernel-opt} with a set of collocation points
$X \subset \overline{\Omega}$ where $X_\Omega \subset X$ denotes the collocation points 
in the interior of $\Omega$ and  $X_{\partial \Omega}$ denotes the collocation points 
on the boundary $\partial\Omega$. Define the fill-distances 
\begin{equation*}
    h_{\Omega} := \sup_{\bx' \in \Omega} \inf_{\bx \in X_\Omega} | \bx' - \bx|, 
    \qquad h_{\partial \Omega} := \sup_{\bx' \in \partial \Omega} \inf_{\bx \in X_{\partial \Omega}} 
    \rho_{\partial \Omega}(\bx', \bx),
\end{equation*}
where $\rho_{\partial \Omega}: \partial \Omega \times \partial \Omega \to \R_+$ is 
the geodesic distance defined on $\partial \Omega$ (see \Cref{sec:AppA}), and set 
 $\bar{h} := \max\{ h_\Omega, h_{\partial\Omega} \}$.
 Then there exists a constant $h_0 >0$ so that 
if $\bar{h} < h_0$ then 
\begin{equation*}
    \| u^\dagger - u^\star \|_{H^s(\Omega)}  \le C \bar{h}^{\gamma} \| u^\star \|_{\mU},
\end{equation*}
where $C >0$ is independent of $u^\dagger$ and $\bar{h}$.
\end{theorem}

\begin{proof}
The proof follows an identical approach to \Cref{thm:PDE-error-bound-no-BC} 
and applies \Cref{thm:abstract-error-analysis} with the appropriate setup. 
We take the operator $\mT: u \mapsto \big( \mP(u), \mB(u) \big)$. 
We then choose the spaces  
$V_1 \equiv H^\ell(\Omega)$, $V_2 \equiv H^k(\Omega) \times H^t(\partial \Omega)$, 
$V_3 \equiv H^{k+\gamma}(\Omega) \times H^{t + \gamma}(\Omega)$, and $V_4 \equiv H^s(\Omega)$
where we equip $V_2$ with the norm $\| (f, g) \|_2 := \| f \|_{H^k(\Omega)} + \| g \|_{H^t(\partial \Omega)}$ 
and  similarly for $V_3$ with the $H^k(\Omega)$ and $H^t(\partial \Omega)$ norms replaced 
by $H^{k+\gamma}(\Omega)$ and $H^{t+\gamma}(\partial \Omega)$ norms.

Analogously to the proof of \Cref{thm:PDE-error-bound-no-BC}, we can verify Conditions 
\ref{cond:A1}, \ref{cond:A3}, and \ref{cond:A4} by the hypothesis of the theorem. Condition \ref{cond:A5}
is also satisfied since $u^\dagger$ is a minimizer of \eqref{PDE-kernel-opt-with-BC} 
and so $\| u^\dagger \|_\mU \le \| u^\star \|_\mU$ as $u^\dagger$ satisfies more relaxed constraints. 

It remains for us to verify \ref{cond:A2}. 
 Repeating the same argument 
as in the proof of \Cref{thm:PDE-error-bound-no-BC}, in the interior of $\Omega$, 
yields the bound 
\begin{equation}\label{disp1}
    \| \mP(u^\dagger) - \mP(u^\star) \|_{H^k(\Omega)} 
    \le C h_\Omega^\gamma \| \mP(u^\dagger) - \mP(u^\star) \|_{H^{k+\gamma}(\Omega)},
\end{equation}
whenever $h_\Omega < h_1$ and $h_1$ is a sufficiently small constant that is independent of $u^\dagger$
and $u^\star$.

Let $\bar{g} = \mB(u^\dagger) - \mB( u^\star)$ which satisfies $\bar{g}(\bx) = 0$ for all $\bx \in X_{\partial \Omega}$
and so 
$\bar{g} \in H^{t+\gamma}(\partial \Omega)$ is zero on the set $X_{\partial \Omega}$. Then \Cref{fuselier-sobolev-bound}
implies the existence of a constant $h_2 >0$ so that whenever $h_{\partial \Omega} < h_2$ we have 
\begin{equation*}
    \| \bar{g} \|_{H^{t}(\partial \Omega)} \le C h_{\partial \Omega}^\gamma \| \bar{g} \|_{H^{t + \gamma}(\partial \Omega)}.
\end{equation*}
Now take $h_0 = \min \{ h_1, h_2\}$ and combine the above bound with \eqref{disp1}, and substitute the 
definition of $\bar{g}$ to get 
\begin{equation*}
\begin{aligned}
\| \mP(u^\dagger) - \mP(u^\star) \|_{H^k(\Omega)} & + \| \mB(u^\dagger) - \mB(u^\star) \|_{H^{t}(\partial \Omega)} \\
& \le C \bar{h}^\gamma 
\big(\| \mP(u^\dagger) - \mP(u^\star) \|_{H^{k+\gamma}(\Omega)} + \| \mB(u^\dagger) - \mB(u^\star) \|_{H^{t+\gamma}(\partial \Omega)}\big),
\end{aligned}
\end{equation*}
whenever $\bar{h} < h_0$. This  verifies \ref{cond:A2} with $\veps \equiv C \bar{h}^\gamma$.
\end{proof}

\begin{remark}\label{remark:time-dependent-PDEs-and-mixed-BC}
    We highlight that our statement of \Cref{thm:PDE-error-bound-with-BC} can easily be extended 
    to PDEs with mixed boundary conditions simply by modifying the norm that is chosen on the boundary, i.e., 
    the spaces $V_2$ and $V_4$, so long as we can prove the requisite stability estimates 
    in condition \ref{assumption-PC-ii}. In particular, this idea will allow us to obtain errors for 
    time-dependent PDEs, cast as a static PDE in an space-time domain $\Omega$ with the initial and boundary 
    conditions imposed as mixed conditions on $\partial \Omega$. In fact, in the case of time-dependent PDEs 
    we do not need to impose the boundary conditions on all of $\partial \Omega$ but only on a subset. 
\end{remark}

We now return to our running example to verify \Cref{assumption:PDE-regularity-with-BC}
for the Darcy flow PDE.

\begin{example}[Nonlinear Darcy flow continued]
Consider the PDE \eqref{Darcy-PDE-example} but this time with the boundary condition 
$u = g$ on $\partial \Omega$
for a function $g \in H^{t+ \gamma}(\partial\Omega)$ with $t > \min\{ 3/2, (d-1)/2 \}$ and $\gamma >0$. Now fix a function $\varphi \in H^{t + \gamma + 1/2}(\Omega)$
so that its trace coincides with $g$ and define $v = u - \varphi$ and observe that 
$u$ solves the above PDE if $v$ solves 
 \begin{equation*}
        \left\{
        \begin{aligned}
          -\Div \big( \exp(a ) \nabla v \big) + \tau'(v)  & = f'  ,  && \bx \in \Omega, \\
          v & = 0, && \bx \in  \partial \Omega,
        \end{aligned}
        \right.
    \end{equation*}
    where we defined $\tau'(v):= \tau( v + \varphi)$ and $f':= f + \Div \big( \exp(a ) \nabla \varphi \big)$. 
    Now observe that the functions $\tau'$ and $f'$ still satisfy the same conditions as $\tau, f$ in 
    \Cref{ex:Darcy-2} and so we obtain existence and uniqueness of the solutions $v$ and in turn 
    $u$.
    
    Now consider two solutions $u_1, u_2$ arising from source terms $f_1, f_2$ and boundary data  
    $g_1, g_2$. Then the error $w = u_1 - u_2$ solves the PDE 
    \begin{equation*}
        \left\{
        \begin{aligned}
          -\Div \big( \exp(a ) \nabla w \big)   & = f + \tau(u_2) - \tau(u_1)   ,  && \bx \in \Omega, \\
          w & = g_1 - g_2, && \bx \in  \partial \Omega,
        \end{aligned}
        \right.
    \end{equation*}
    By standard stability results for linear elliptic PDEs \cite[Thm.~4.18]{mclean2000strongly} we have 
    \begin{equation*}
        \| w \|_{H^2(\Omega)} \le B( \| f_1 - f_2 \|_{L^2(\Omega)} + \| \tau( u_1) - \tau(u_2) \|_{L^2(\Omega)} 
        + \| g_1 - g_2 \|_{H^{3/2}(\Omega)} ) 
    \end{equation*}
    We can now repeat the same argument as in the final steps of \Cref{ex:Darcy-2} to get the bound 
        \begin{equation*}
     \| w\|_{H^2(\Omega)} \le  \frac{B}{1-B\beta} \left( \| f_1 - f_2 \|_{H^2(\Omega)} 
     + \| g_1 - g_2 \|_{H^{3/2}(\Omega)}\right),
     \end{equation*}
     which verifies \Cref{assumption:PDE-regularity-with-BC}(ii) with $s = 2$ provided that $\beta B < 1$.
\end{example}

\subsection{The Case of Parametric PDEs}\label{subsec:Parametric-PDE-error-bounds}
We now consider the setting of the parametric PDE \eqref{abstract-parametric-PDE}. Our error 
estimates can be viewed as further extending  \Cref{thm:PDE-error-bound-with-BC} with 
additional assumptions due to the fact that we will need to approximate the solutions
on the set $\Upsilon = \Omega \times \Theta$
as well as its relevant boundary which needs to be sufficiently regular for us to apply 
\Cref{fuselier-sobolev-bound}. Beyond this technical point, the statement and proof of the 
result for parametric PDEs is identically to PDEs with boundary conditions and so we state our results 
succinctly, starting with the requisite assumptions on the parametric PDE. 

\begin{assumption}\label{assumption:parametric-PDE-convergence}
The following conditions hold: 
\begin{enumerate}[label=(D\arabic*)]
    \item $\Omega \subset \R^d$ and $\Theta \subset \R^p$ are compact sets such that 
    $\partial \Omega$ and $\partial \Theta$ are smooth Riemannian  manifolds of dimensions $d-1$ and $p -1$
    respectively. 
    \label{assumption:PPB-i}

    \item {\it (Stability of the parametric PDE)} 
    There exist $\gamma>0$
    and $k, t \in \mathbb{N}$ satisfying $(d+ p)/2 < k+\gamma $ and $(d + p -1)/2 < t+\gamma$, and 
    Banach spaces $V_1$ and $V_4$ 
    so that for any $r > 0$
    it holds that 
    \label{assumption:PPB-ii}
    \begin{equation}\label{eq:parametric-PDE-stability-estimate-backward}
        \begin{aligned}
        & \| u_1 - u_2 \|_{1}  \\
        & \quad \le  C \big( \| \mP(u_1) - \mP(u_2) \|_{H^{k}(\Upsilon)} 
        + \| \mB(u_1) - \mB(u_2) \|_{H^{t}(\partial\Upsilon)} 
        \big)   
        \quad \forall u_1, u_2 \in B_r(V_1),
        \end{aligned}
    \end{equation}
    \begin{equation}\label{eq:parametric-PDE-stability-estimate-forward}
    \begin{aligned}
    \| \mP(u_1) - \mP(u_2) \|_{H^{k+\gamma}(\Upsilon)} 
    & + \| \mB(u_1) - \mB(u_2) \|_{H^{t + \gamma}(\partial \Upsilon)}  \\
     & \le C \| u_1 - u_2 \|_{4},  
         \quad \forall u_1, u_2 \in B_r( V_4), 
    \end{aligned}
    \end{equation}
    where $C = C(r) >0$ is a constant independent of the $u_i$.
    \item $\mU$ is continuously 
    embedded in $V_4$. 
    \label{assumption:PPB-iii}
\end{enumerate}
\end{assumption}


Unlike Assumptions~\ref{assumption:PDE-regularity} and \ref{assumption:PDE-regularity-with-BC}
here we left the function spaces $V_1$ and $V_4$ as generic Banach spaces of functions $u : \Upsilon \mapsto \R$ 
since, for parametric PDEs, we can often obtain the desired stability results in non-standard norms, 
such as the mixed norm in \Cref{paramexdfl} below, as opposed to the Sobolev norms 
used for the non-parametric PDE setting. More generally, one may also impose $V_2,V_3$ to be generic Banach spaces rather than the standard Sobolev spaces. The Sobolev space setting suffices for applications in this paper.

With the above assumptions we can now present our main result for the parametric PDE setting. 
The proof is omitted since it is identical to that of \Cref{thm:PDE-error-bound-with-BC} except that 
(1) the argument on  $\partial \Omega$ is now repeated for $\partial \Upsilon = \partial \Omega \times \Theta$ 
which is in general a smooth manifold with boundary but this modification does not affect any of the steps 
in the proof, and (2) the results are stated in terms of the norm on the space $V_1$.

\begin{theorem}\label{thm:parametric-PDE-error-bound}
Suppose \Cref{assumption:parametric-PDE-convergence} is satisfied and let 
 $u^\star \in \mU$ denote the unique strong solution of \eqref{abstract-parametric-PDE}.
Let $u^\dagger$ be a minimizer of \eqref{parametric-PDE-kernel-opt} with a set of collocation points
$S \subset \Upsilon \cup \partial \Upsilon$ where $S_\Upsilon \subset S$ denotes the collocation points 
in the interior of $\Upsilon$ and $S_{\partial \Upsilon}$ denotes the collocation points 
on the boundary $\partial\Upsilon$. Define the fill-distances 
\begin{equation*}
    h_{\Upsilon} : = \sup_{\bs' \in \Upsilon} \inf_{\bs \in S_\Upsilon} | \bs' - \bs|, 
    \qquad 
    h_{\partial\Upsilon} : = \sup_{\bs' \in \partial\Upsilon} \inf_{\bs \in S_{\partial\Upsilon}} 
    \rho_{\partial \Upsilon} (\bs', \bs),
\end{equation*}
where $\rho_{\partial \Upsilon}: \partial \Upsilon \times \partial \Upsilon \to \R_+$ is the geodesic 
distance defined on $\partial \Upsilon$ (see \Cref{sec:AppA}),
and set $\bar{h} := \max\{ h_\Upsilon, h_{\partial\Upsilon} \}$.
 Then there exists a constant $h_0 >0$ so that 
if $\bar{h} < h_0$ then 
\begin{equation*}
    \| u^\dagger - u^\star \|_{1}  \le C \bar{h}^{\gamma} \| u^\star \|_{\mU},
\end{equation*}
where $C >0$ is independent of $u^\dagger$ and $\bar{h}$.
\end{theorem}

We end this section by returning to our example of the Darcy flow PDE but this time in the setting  where
the coefficient $a$ and the source $f$ are dependent on a finite dimensional parameter $\btheta$.
We will show that \Cref{assumption:parametric-PDE-convergence} can be verified 
in this case, with $V_1$ and $V_4$ taken as Banach spaces with mixed 
Sobolev and $L^2$ norms, and so \Cref{thm:parametric-PDE-error-bound} is applicable. 

\begin{example}[1D Parametric Darcy flow PDE] \label{paramexdfl}
Consider the parametric elliptic PDE
 \begin{equation*}
        \left\{
        \begin{aligned}
          -\Div \big( A(\bx,\btheta) \nabla u \big)   & = f(\bx,\btheta) && \bx \in \Omega, \\
          u & = 0, && \bx \in  \partial \Omega,
        \end{aligned}
        \right.
    \end{equation*}
    over a compact domain $\Omega$  and $\btheta \in \Theta$ where both $\Omega$ and $\Theta$ are assumed to 
    satisfy condition~\ref{assumption:PPB-i}; e.g., take $\Omega$ and $\Theta$ to be unit balls. 
    For simplicity we are ignoring the boundary operator in this case and imposing 
    homogeneous Dirichlet boundary conditions.
    In this example we assume $a$ is smooth in both $\bx$ and $\btheta$, and there exists $m, M >0$ such that $m \leq A(\bx,\btheta)\leq M$. As a concrete example we may take $A(\bx, \btheta) = \sum_{j=1}^p \theta_j \psi_j(\bx)$ 
    where the $\psi_j$ are a set of smooth functions on $\overline{\Omega}$ that are uniformly bounded from 
    below. 

First, the boundness of the operator $\mathcal{P}$ is straightforward to obtain, since $a$ is smooth and its derivatives will be bounded in the bounded domain $\Omega \times \Theta $. More precisely, for any $\gamma >0$, since $f =  -\Div \big( A(\bx,\btheta) \nabla u \big) = - \nabla_{\bx} A \cdot\nabla_{\bx} u - A \Delta_{\bx} u$, there exists some constant $C$ independent of $u$ and $f$ such that
\begin{equation*}
    \|f\|_{H^{\gamma}(\Omega \times\Theta)} \leq C\|u\|_{H^{\gamma+2}(\Omega \times\Theta)}\, .
\end{equation*}
Due to the linearity of the equation, by replacing $u$ by $u_1-u_2$ and noting that $f = \mP u = \mP u_1 -\mP u_2$, we obtain the forward stability
\begin{equation}
\label{eqn: parametrized Darcy, forward stab}
    \|\mP u_1 - \mP u_2\|_{H^{\gamma}(\Omega \times\Theta)} \leq C\|u_1 - u_2\|_{H^{\gamma+2}(\Omega \times\Theta)}\, .
\end{equation}


For the backward stability estimate, via intergation by parts, we have
\begin{equation*}
\begin{aligned}
  \int_{\Upsilon} A (\bx, \btheta) |\nabla_{\bx} u(\bx,\btheta)|^2 \, {\rm d} \bx{\rm d}\btheta 
  &= \int_{\Upsilon} u(\bx,\btheta)f(\bx,\btheta)\, {\rm d} \bx{\rm d}\btheta\\
  & \leq \int_{\Theta} \|u(\cdot,\btheta)\|_{L^2(\Omega)}\|f(\cdot,\btheta)\|_{L^2(\Omega)}\, {\rm d}\btheta\\
  & \leq C_0 \int_{\Theta} \|\nabla_\bx u(\cdot,\btheta)\|_{L^2(\Omega)}\|f(\cdot,\btheta)\|_{L^2(\Omega)}\, {\rm d}\btheta\\
  & \leq C_1 \|u\|_{L^2(\Theta, H^1_0(\Omega))} \|f\|_{L^2(\Theta, L^2(\Omega))}\, ,
\end{aligned}
\end{equation*}
where in the first and third inequalities, we used the Cauchy-Schwarz inequality; in the second inequality, we used the Poincar\'e inequality as $u(\cdot, \btheta)$ is zero on $\partial \Omega$. Here we used the notation: 
\[\|u\|_{L^2(\Theta, H^1_0(\Omega))}^2 :=\int_\Theta \|u(\cdot,\btheta)\|^2_{H^1_0(\Omega)}\, {\rm d}\btheta\, 
\quad \text{and} \quad 
\|f\|^2_{L^2(\Theta, L^2(\Omega))} := \int_\Theta \| f (\cdot, \btheta) \|_{L^2(\Omega)}^2 \dd \btheta.\]
Note that the $L^2(\Theta, L^2(\Omega))$ norm is also equivalent to the $L^2(\Omega \times \Theta)$ norm.
Now, using the bound on $A$, we obtain that there exists a constant $C$ such that
\[ \|u\|_{L^2(\Theta, H^1(\Omega))}\leq C\|f\|_{L^2(\Omega \times \Theta)}\, .\]
Similar to the proof for the forward stability, the backward stability follows by the linearity of the equation. We have
\begin{equation}
\label{eqn: parametrized Darcy, backward stab}
    \|u_1 - u_2\|_{L^2(\Theta, H^1(\Omega))}\leq C\|\mP u_1 -\mP u_2\|_{L^2(\Omega \times \Theta)}.
\end{equation}
Thus it follows that we can verify condition~\eqref{eq:parametric-PDE-stability-estimate-backward} 
with the norm $\| \cdot \|_1 \equiv \| \cdot \|_{L^2(\Theta, H^1_0(\Omega) )}, \|\cdot\|_4 = \|\cdot\|_{H^{\gamma+2}(\Omega\times\Theta)}$, and $k = 0$.

\end{example}

\subsection{Bounding Fill-distances}
Our bounds in \Cref{thm:PDE-error-bound-no-BC,thm:parametric-PDE-error-bound} 
are given in terms of the fill distances $h$ of our collocation points.
In this section, we provide an upper bound of these fill-in distances
in terms of the number of collocation points, under the assumptions that the points are randomly drawn 
according to uniform distributions both in the interior of 
the domains and their pertinent boundaries. Throughout this section we only consider 
the case of non-parametric PDEs, hence we work with $\Omega$, assumed to be a compact subset of 
$\R^d$ with boundary $\partial \Omega$ which is a compact smooth manifold of dimension $d-1$. We focus on the non-parametric setting for simplicity and our results can easily be extended to the 
parametric PDE setting by simply replacing $\Omega$ with $\Upsilon$ as a compact subset of $\R^{d + p}$.

\begin{proposition}
\label{prop: fill-in dist bound}
Suppose we sample $M_{\Omega}$ points in $\Omega$ and $M_{\partial\Omega}$ points on $\partial \Omega$, uniformly with respect to the canonical volume and surface measures. Let $\delta >0$. Then, with probability at least $1-\delta$, the fill-in distances $h_{\Omega}$ and $h_{\partial\Omega}$ satisfy
\[h_{\Omega} \leq C\left(\frac{\log(M_{\Omega}/\delta)}{M_{\Omega}}\right)^{1/d}, \quad h_{\partial\Omega} \leq C\left(\frac{\log(M_{\partial\Omega}/\delta)}{M_{\partial\Omega}}\right)^{1/{(d-1)}}, \]
where $C$ is a constant independent of $M_{\Omega}, M_{\partial\Omega}$ and $\delta$.
\end{proposition}
The proof of \Cref{prop: fill-in dist bound} can be found in \Cref{sec: Bounds on fill-in distance}.

Let's combine \Cref{prop: fill-in dist bound} and previous error estimates to get error bounds regarding the number of collocation points. In the case of \Cref{thm:PDE-error-bound-no-BC} where there is no boundary, we get
\[\| u^\dagger - u^\star \|_{H^{s}(\Omega)} \le C\left(\frac{\log(M_{\Omega}/\delta)}{M_{\Omega}}\right)^{\gamma/d} \| u^\star \|_\mU, \]
while in the case of \Cref{thm:PDE-error-bound-with-BC} where boundary is considered, we have 
\[\| u^\dagger - u^\star \|_{H^{s}(\Omega)} \leq C\left( \left(\frac{\log(M_{\Omega}/\delta)}{M_{\Omega}}\right)^{\gamma/d} + \left(\frac{\log(M_{\partial\Omega}/\delta)}{M_{\partial\Omega}}\right)^{\gamma/(d-1)} \right) \| u^\star \|_\mU.  \]
More generally, we note that the bounds in \Cref{prop: fill-in dist bound} can be applied to the abstract setting in \Cref{thm:abstract-error-analysis} when $\epsilon$ depends on the fill-in distance. 

If $s$ and $\gamma$ are appropriately chosen such that the required assumptions hold, and $\gamma \geq d/2$, then the convergence rate is at least as fast as the Monte Carlo rate, for uniformly sampled collocation points. There is no curse of dimensionality in this case.
 

\section{Numerical Experiments}\label{sec:numerics}
In this section, we study several numerical examples to demonstrate the interplay between the dimensionality of the problem and the regularity of the solution. Our
theory demonstrates that this interplay is central to determining the convergence rate,
and hence accuracy, of the methodology studied in this paper.

In \Cref{exp sec: high dim PDEs}, we consider a high dimensional elliptic PDE with smooth solutions. By varying the dimension of the problem and the frequency of the solution, we demonstrate dimension-benign convergence rates, and in particular the accuracy is better when the frequency of the solution is lower. In \Cref{exp sec: Parametric PDEs}, we consider a high dimensional parametric PDE problem to illustrate the importance of choosing kernels that adapted to the regularity of the solution. In \Cref{section: High D HJB eqn}, we present a high dimensional Hamilton-Jacobi-Bellman (HJB) equation, which goes beyond our theory and demonstrates the interplay between dimensionality and regularity.

\subsection{High Dimensional PDEs}
\label{exp sec: high dim PDEs}
Consider the variable coefficient nonlinear elliptic PDEs
\begin{equation}\label{eq:yifan-darcy}
      \left\{
      \begin{aligned}
        - \nabla \cdot ( A \nabla  u ) + u^3  &= f, && \text{in } \Omega, \\
        u &= g, && \text{on } \partial \Omega.  
      \end{aligned}
      \right.
  \end{equation}
We set $A(\bx) =  \exp \left( \sin \left(\sum_{j=1}^d \cos(x_j)\right) \right)$, and the ground truth solution \[u^\star(\bx;\beta) = \exp(\sin(\beta \sum_{j=1}^d \cos(x_j)))\, ,\]
where we have a parameter $\beta$ to control the frequency of $u$. The right hand side and boundary data are obtained using $A$ and $u^\star$.

In the experiment, we choose the domain $\Omega$ to be the unit ball in $\R^d$ for $d = 2,3,\ldots,6$. We sample $M_{\Omega} = 1000,2000,4000,8000$ points uniformly in the interior, and respectively $M_{\partial \Omega} = 200,400,800, 1600$ points uniformly on the boundary.

After selecting the kernel function, the number of iteration steps in our algorithm is set to be $3$ with initial solution $0$. We sample another set of $M_{\Omega}$ test points and evaluate the $L^2$ error of the solution on these points. The results are averaged over $10$ independent draws of the uniform collocation points.

In the first experiment, we choose the Mat\'ern kernel with $\nu = 7/2$ and with lengthscale $\sigma = 0.25\sqrt{d}$. We choose $\beta = 1, 4$, to compare the convergence given ground truth with different frequencies. The results are shown in \cref{fig-high-d-example-vary-beta}. It is clear that when $\beta$ is small, the accuracy is better. The slopes of convergence curves also have a tendency to improve for $d\geq 3$ if we increase $\beta$.
 \begin{figure}[htp]
     \centering
     \includegraphics[width=7.5cm]{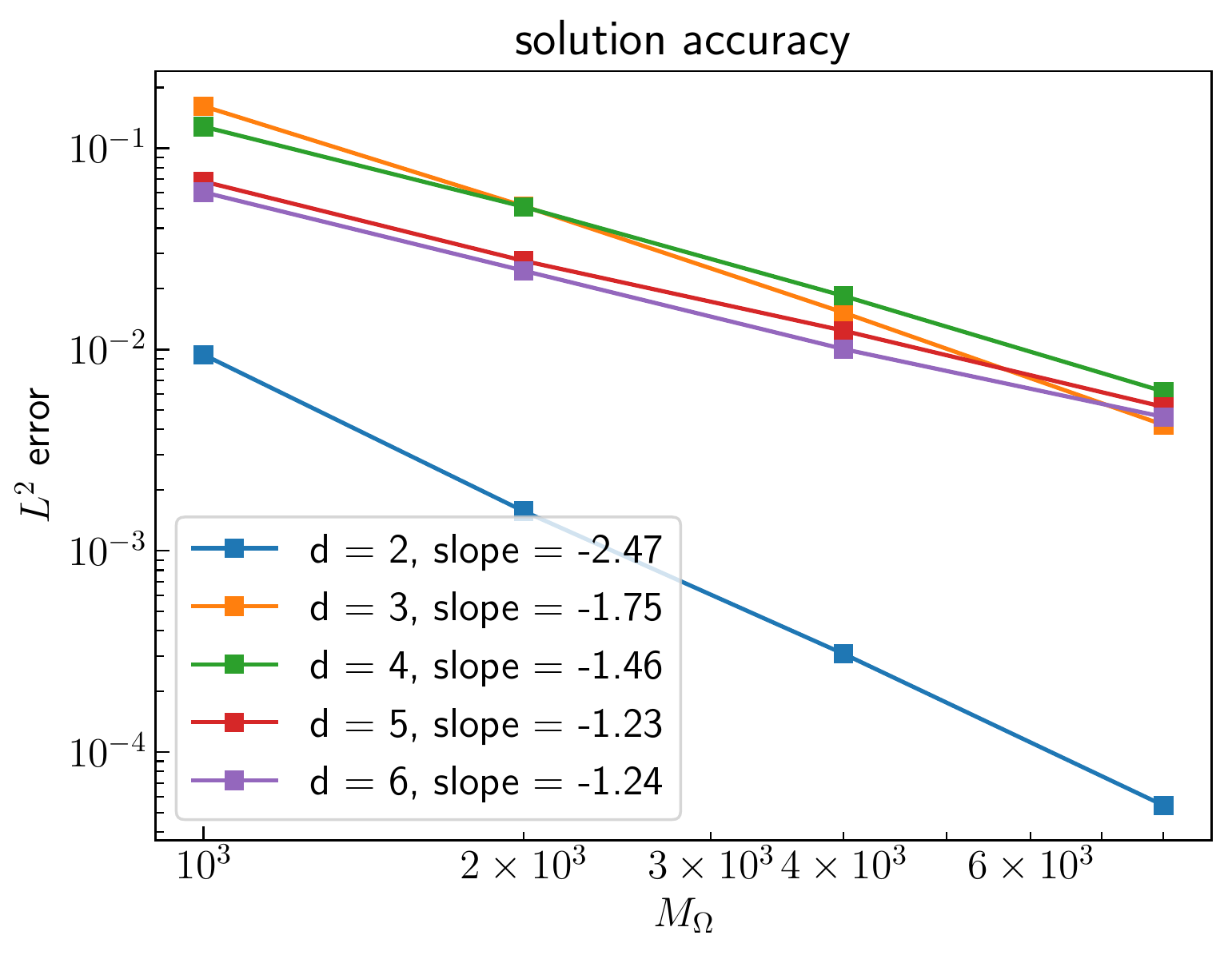}
     \includegraphics[width=7.5cm]{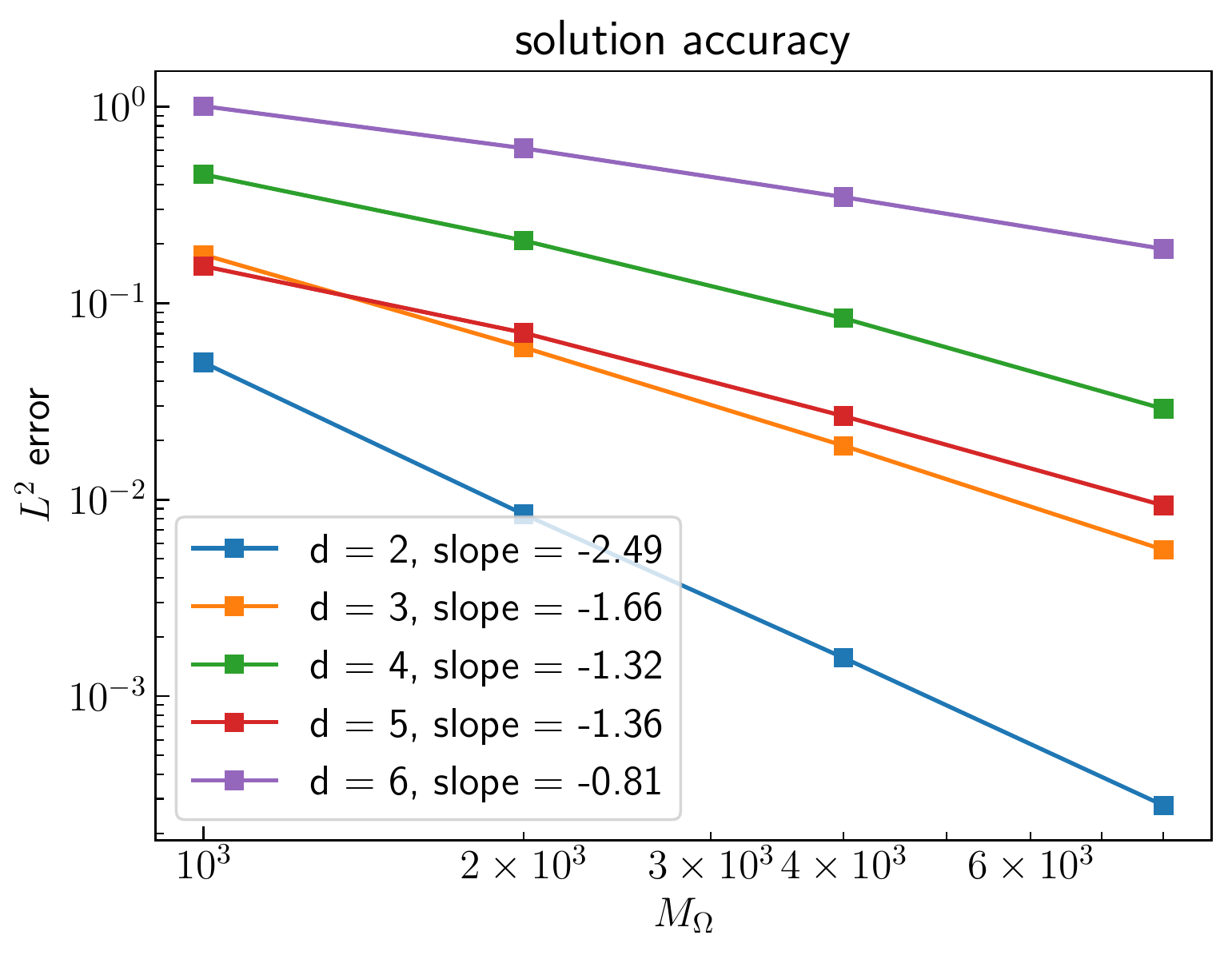}
     \caption{
 $L^2$ test errors of  solutions to Problem~\eqref{eq:yifan-darcy} as a function of the number of collocation points. 
     Left: $\beta = 1$; right: $\beta =4$. In both cases, we choose Mat\'ern kernel with $\nu = 7/2$. 
     Reported slopes in the legend denote empirical convergence 
     rates.
     }
     \label{fig-high-d-example-vary-beta}
 \end{figure}
 
In the second experiment, we fix $\beta = 4$, and choose the Mat\'ern kernel with $\nu = 5/2, 9/2$ and with  lengthscale $\sigma = 0.25\sqrt{d}$. Results are shown in \Cref{fig-high-d-example-vary-kernel}. Comparing $\nu = 5/2,9/2$ and $\nu=7/2$ in the last example, we observe that increasing $\nu$ leads to faster convergence. This is due to the fact that the true solution is smooth. In dimension $d=2$, we can identify the exact convergence rate as $\nu-1$. In all dimensions, the rate is faster than the Monte Carlo rate. 
We observe that the regularity of the 
solution softens the effect of the 
curse of dimensionality, i.e., convergence 
rates are better in higher dimensions when $\beta$ is 
smaller.
 \begin{figure}[ht]
     \centering
     \includegraphics[width=7.5cm]{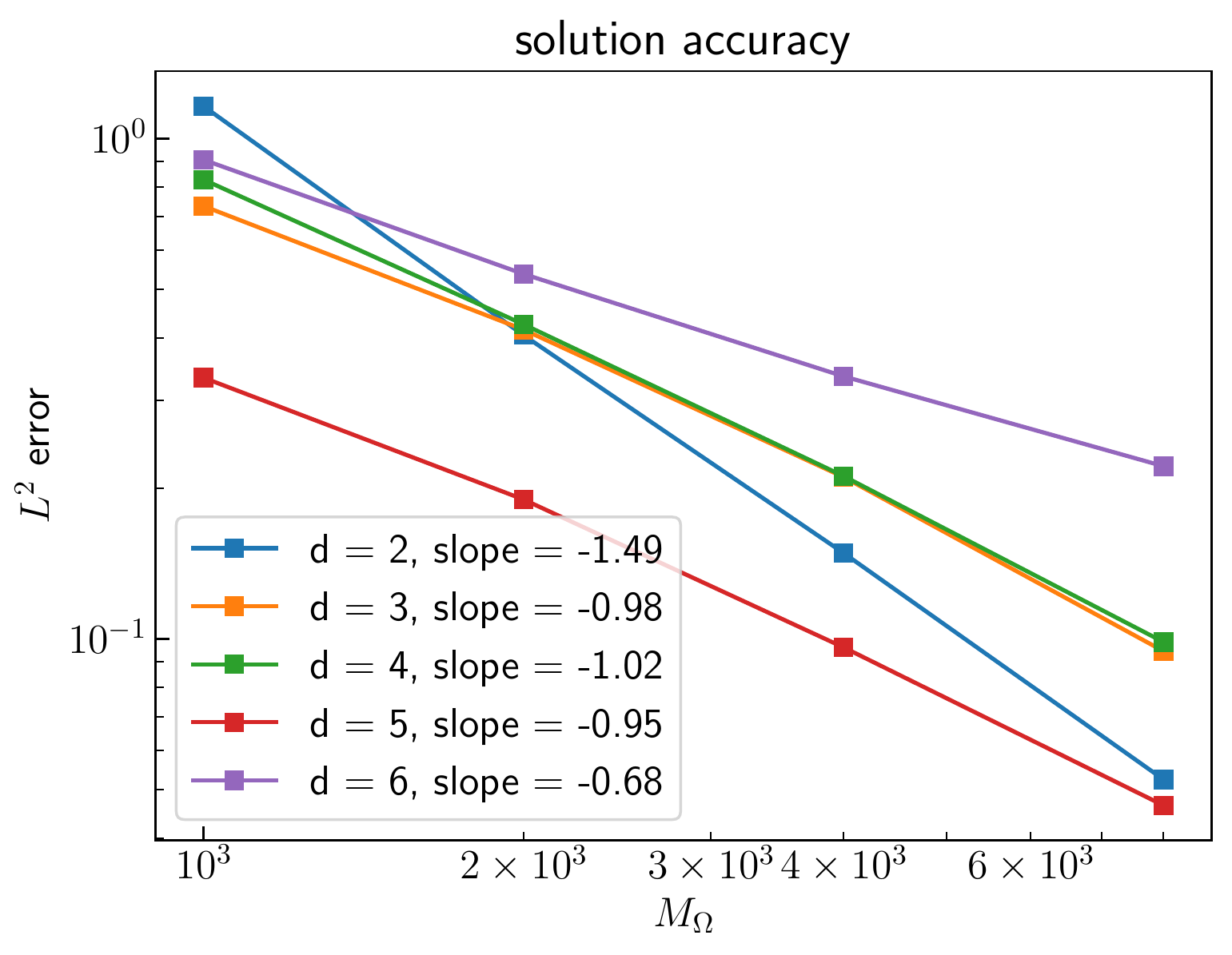}
     \includegraphics[width=7.5cm]{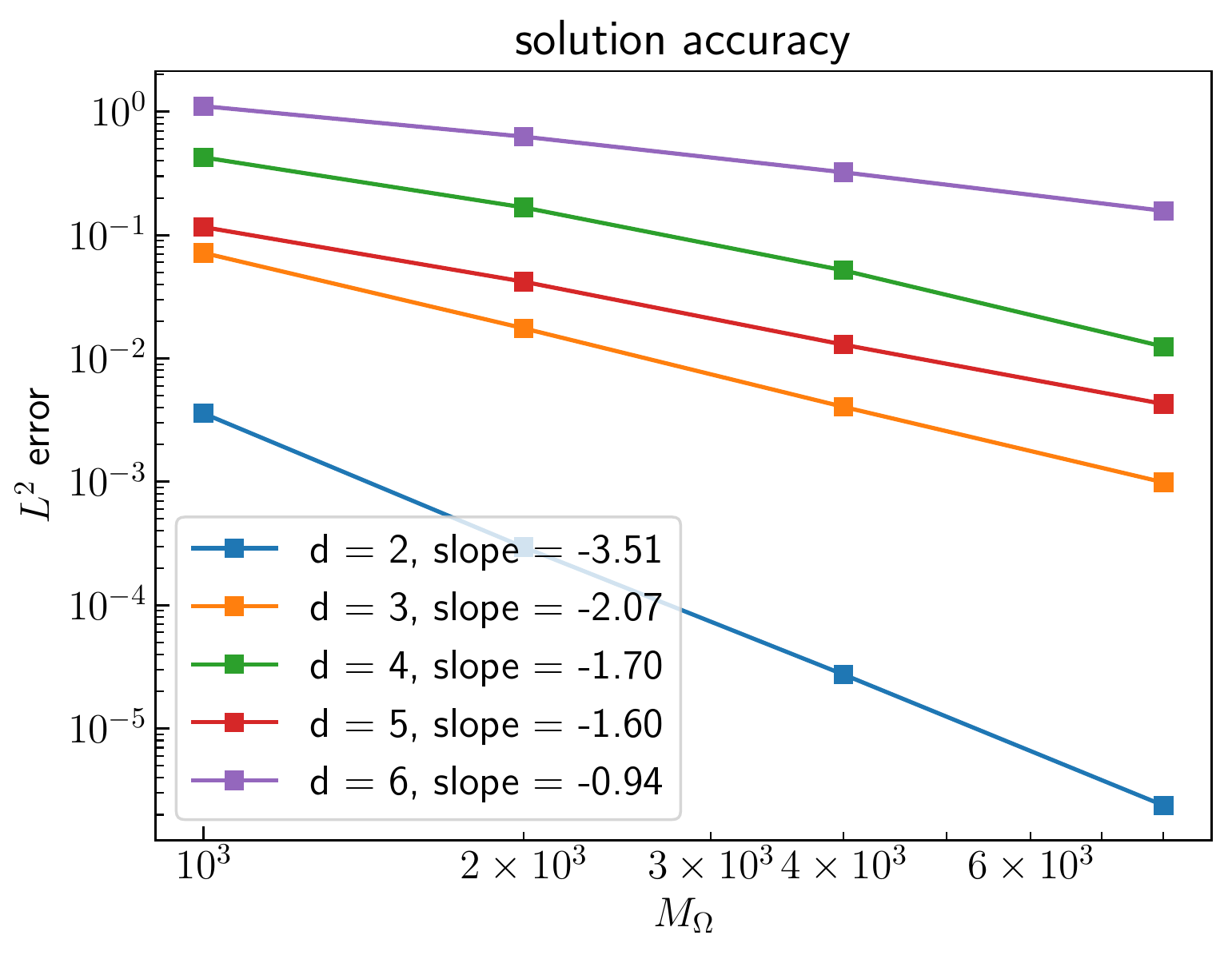}
     \caption{
     $L^2$ test errors of  solutions to Problem~\eqref{eq:yifan-darcy} as a function of the number of collocation points with $\beta = 4$. 
     Left: Mat\'ern kernel with $\nu = 5/2$; right: Mat\'ern kernel with $\nu = 9/2$.      Reported slopes in the legend denote empirical convergence 
     rates.}
     \label{fig-high-d-example-vary-kernel}
 \end{figure}

\subsection{Parametric PDEs}
\label{exp sec: Parametric PDEs}
We consider a parametric version of the linear ($\tau = 0$) darcy flow problem in \cref{ex:darcy-flow}:
      \begin{equation}\label{eq:parameq1}
        \left\{
        \begin{aligned}
          -\Div \big( \exp(a(\bx, \btheta) ) \nabla u \big) (\bx)  & = f(\bx), && \bx \in \Omega, \\
          u(\bx) & = g(\bx), && \bx \in  \partial \Omega,
        \end{aligned}
        \right.
    \end{equation}

Following the general form \cref{abstract-parametric-PDE}, we aim to obtain the solution as a function taking values in the product space $\Upsilon$. \cref{eq:parameq1} can be rewritten in terms of $\bs = (\bx, \btheta)$ with new forcing terms $\hat f$ and $\hat g$ depending only on the first coordinate of $\bs$

\begin{equation}\label{eq:parameq2}
    \left\{
    \begin{aligned}
      -\Div_{\bx} \big( A(\bs) \nabla_{\bx} u(\bs) \big) (\bs)  & = \hat{f}(\bs) = f(\bx), && \bs \in \Upsilon, \\
      u(\bs) & = \hat{g}(\bs) = g(\bx), && \bs \in \partial \Upsilon.
    \end{aligned}
    \right.
\end{equation}
Recall that we defined $\partial \Upsilon = 
\partial \Omega \times \Theta$.
For our numerical example, we let $d = 1$ and vary $p$. We set $A(x, \theta) = 2+\theta_0 + \sum_{j = 1}^{p}\frac{\theta_j}{j^k}\sin (\pi x + j)$, $f(x) = x$ and $g(x) = 0$, a similar setting as in \cite{Chkifa2012}. We choose $\Omega = [0,1]$, and $\Theta =[0,1]^{p}$, for $p = 2, 3, \ldots, 6$. Note $1 \leq A(\bs) \leq 4$ since the sum is in $[-1,1]$ for all $p$ and $\theta \in \Theta$, matching the setting of \cref{paramexdfl}. 

We sample different $M_{\Omega}$ points uniformly in the interior, and $M_{\partial \Omega} = M_\Omega / 10$ points uniformly on the boundary of $x$. We do two experiments with different choices of kernel, in the first (\cref{fig:parametricnumerics}, left), a vanilla Gaussian kernel with different length scales for the $x$ and $\theta$ dimension, and with a scaling of the length scale in $\theta$ proportional to $\sqrt{p}$. In the second one (\cref{fig:parametricnumerics}, right), we adapt the Gaussian kernel to the decay in $A(x, \theta)$, by including the decay of $1/j^k$ in the norm in $\theta$ space used by the kernel. We see significant improvement in test error using this adaptation in high dimensions, which suggests future research directions of kernel adaptation to the specific form of the PDE.  
In all cases, we use a cross-validation procedure for hyperparameter tuning and we observe the average $L^2$ test error on an independent set of test points for different values of $p$ and $M_\Omega$.
Since $d=1$ we computed our ground truth solution 
by numerically integrating Equation \cref{eq:parameq2}
using quadrature.

As mentioned, this problem was also explored by \cite{Chkifa2012}, in which sparse multivariate polynomials are used to estimate the solution with a rate independent of the number of parameters, provided the decay of the coefficient functions is large enough (in $\ell^p$ for some $0 < p < 1$). While this assumption is satisfied in this example, our method's convergence rate greatly depends on the dimension of $\theta$ when the kernel is not adapted to the particular equations and coefficients $A(x, \theta)$. Our results indicate improvement in 
the dependence of 
convergence rates on dimension when the kernel is adapted 
to the regularity of $A$. 
It remains open whether our kernel based approach (which is not specific to parametric equations) can achieve the same dimension independent convergence rates as the ones in \cite{Chkifa2012} (which apply even in the countably infinite dimensional case and which they refer to as breaking
the curse of dimensionality) for parametric elliptic PDEs with rapidly decreasing parametric dependence as specified above (this assumption implies a finite number of effective parameters).

\begin{figure}
    \centering
    \includegraphics[width = 15 cm]{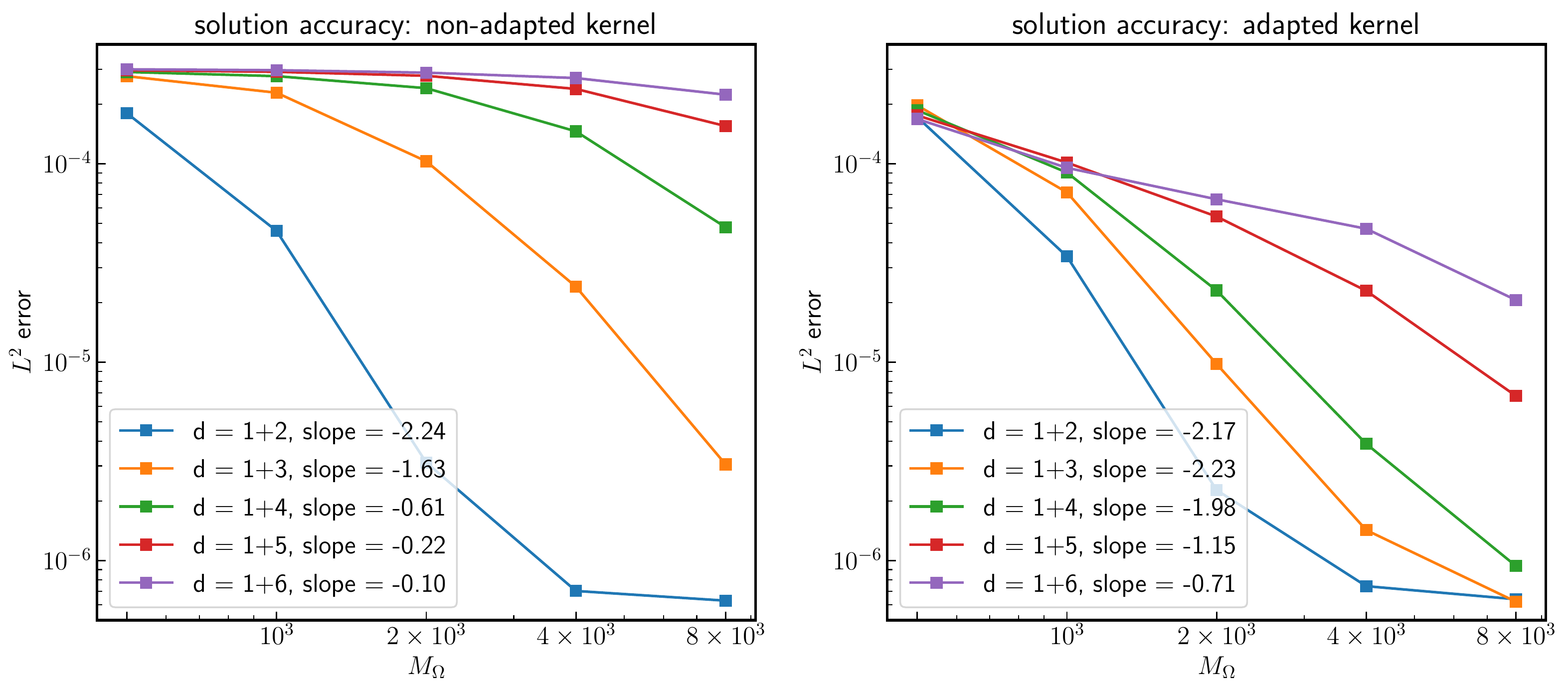}

    \caption{ $L^2$ test error of solutions to Problem~\eqref{eq:parameq2} as a function of the number of collocation points. Left: vanilla gaussian kernel; Right: Gaussian kernel adapted to the regularity of $A$. Reported slopes in the legend denote empirical convergence 
     rates.}
    \label{fig:parametricnumerics}
\end{figure}

\subsection{High Dimensional HJB Equation}
\label{section: High D HJB eqn}
 Consider a prototypical HJB equation:
\begin{equation}\label{eq:yifan-HJB}
\begin{aligned}
  (\partial_t + \Delta)V(\bx,t)-|\nabla V(\bx,t)|^2 &= 0\\
  V(\bx,T) & = g(\bx)\, ,
\end{aligned}
\end{equation}
where, $g(\bx) = \log(\frac{1}{2}+\frac{1}{2}|\bx|^2), \bx \in \mathbb{R}^d, t \in [0,T]$. We are interested in solving $V(\bx_0, 0)$ for some $\bx_0 \in \mathbb{R}^d$.
    We adopt the stochastic differential equation (SDE) formula for representing the solution of the PDEs, following \cite{weinan2017deep,richter2021solving}. More specifically, consider the SDE 
    \begin{equation}
    \label{eqn: forward SDE}
        dX_s = \sqrt{2}dW_s,\ X_0 = \bx_0 \, .
    \end{equation}
    We define $Y_s = V(X_s,s), Z_s = \sqrt{2}\nabla V(X_s,s)$. By Ito's formula, one obtains
    \begin{equation}
    \label{eqn: backward SDE}
        d Y_s = \frac{1}{2}|Z_s|^2 ds +Z_s\cdot dW_s\, .
    \end{equation}
The strategy is to integrate the above SDE backward to $Y_0$. An 
implicit\footnote{Implicit because are integrating backwards in time.} 
Euler discretization from time $t_{n+1}$ to $t_n$ ($\Delta t = t_{n+1}-t_n$) leads to the following equation:
\begin{equation}
    \label{eqn: HJB, one step}
    V(X_{t_{n+1}},t_{n+1}) = V(X_{t_n},t_n) + |\nabla V(X_{t_n},t_n)|^2 \Delta t +\sqrt{2}\nabla V(X_{t_n},t_n) \cdot \xi_{n+1}\sqrt{\Delta t}\, .
\end{equation}

Algorithmically, we sample $J$ different paths of the forward SDE in \eqref{eqn: forward SDE}, namely $X^{(j)}_{t_n}, 1\leq j \leq J$, using the Euler–Maruyama scheme. Then, backward in time, we apply our kernel method, namely to solve the following optimization problem 
\begin{equation}
\label{eqn: HJB optim}
    \left\{
      \begin{aligned}
        \minimize_{u \in \mU}  \quad & \| u \|_\mU \\
        \st \quad  & u(X^{(j)}_{t_n},t_n) + |\nabla u(X^{(j)}_{t_n},t_n)|^2 \Delta t + \sqrt{2}\nabla u(X_{t_n}^{(j)},t_n) \cdot \xi_{n+1}\sqrt{\Delta t} = V(X_{t_{n+1}}^{(j)},t_{n+1})
      \end{aligned}
      \right.
    \end{equation}
to get the solution $V(\cdot, t_n)$, assuming $V(\cdot, t_{n+1})$ has been solved. Iterating this process, we end up with the solution $V(\bx_0, 0)$. We can understand the algorithm as applying our kernel method iteratively with the sample path as the collocation points.


Experimentally, we consider $d =100$ as in \cite{weinan2017deep,richter2021solving}. We aim to solve $V(\bx_0,0)$ for $\bx_0 = 0$. The ground truth is $V(\bx_0,0) = 4.589992$ provided in \cite{weinan2017deep}. We sample $J = 2000$ paths from $\bx_0$ and choose the  inverse quadratic kernel 
$k(\bx,\by;\sigma) = \left(\frac{\|\bx-\by\|^2}{2d\sigma^2}+1 \right)^{-1}\, $. We use the ``linearize-then-optimize" approach to compute an approximate solution to \eqref{eqn: HJB optim}.  The nugget term is set to be $\eta = 10^{-3}$. The result is shown in \Cref{experiments: HJB}. 
\begin{table}[H]
\centering
\begin{tabular}{llllll}
\hline
$\sigma$          & 10     & 25     & 50     & 100 & 200   \\ \hline
Computed solution $V(\bx_0,0)$          & 5.6042 & 4.6366 & 4.6039 & 4.6021 & 4.6021\\
Relative accuracy & 22.10\%      &   1.0154\%     &   0.303\%      & 0.2638\%  & 0.2638\%      \\ \hline
\end{tabular}
\caption{Numerical results for  the HJB equation
\eqref{eq:yifan-HJB}, computing the quantity 
$V(\bx_0, 0)$.
}
\label{experiments: HJB}
\end{table}
We observe that a suitable choice of the lengthscale of the kernel is crucial to obtain an accurate solution. Compared to the relative accuracy of $0.171\%$ (reported in \cite{richter2021solving}) using neural networks (DenseNet like architecture with 4 hidden layers) to solve \eqref{eqn: HJB, one step}, the accuracy of using kernel methods with a simple quadratic kernel is comparable. Moreover, the lengthscale of the kernel is very large, indicating that the solution behavior of this HJB equation is very smooth;  similar “blessings of dimensionality” have been reported and discussed in \cite{richter2021solving}, where they used a constant function (and the terminal function $g$) as ansatz to solve \eqref{eqn: HJB, one step} and obtained very high accuracy\footnote{We anticipate that using the feature map perspective of  kernel methods with constants and $g$ as features will achieve a similar accuracy as in \cite{richter2021solving}. We did not pursue this here to avoid using strong prior information on the solution beyond regularity.}. Thus, this HJB example in dimension $100$ demonstrates again the trade-off between the smoothness of the solution and the curse of dimensionality. 


\section{Conclusions}\label{sec:conclusions}

In this paper, we conducted an error analysis of GP and kernel based methods for solving PDEs. We provided convergence rates under the assumptions that (1) the solution belongs to the RKHS which is embedding to some Sobolev space of sufficient regularity, and (2) the underlying forward and inverse PDE operator is stable in corresponding Sobolev spaces.

Our analysis relies on the crucial minimizing norm property of the numerical solution in the kernel/GP methodology.
The analysis could be seamlessly generalized to the function class of NNs and other norms such as non-quadratic norms if we can formulate the training process as a minimization problem over the related norm. 

We emphasize that our convergence rates hold for the exact minimizer of the minimization problem. In practice, finding such a minimizer algorithmically can be a separate and challenging problem. Our numerical experience suggests that Gauss-Newton iterations usually perform well, and typically, 2-5 iterations are sufficient for convergence. Therefore, we can combine the error analysis in this paper and the fast implementation of the algorithm in \cite{chen2023sparse} to obtain a near-linear complexity solver for nonlinear PDEs with rigorous accuracy guarantee.

It is worth mentioning that this paper focuses only on analyzing the MAP estimator within the GP interpretation. Exploring the posterior distribution of the GP can provide a means for quantifying uncertainty in the solution. In particular, analyzing the posterior contraction is an interesting direction for future research.

{
    
    

\section*{Acknowledgments}
The authors gratefully acknowledge support by  the Air Force Office of Scientific Research under MURI award number FA9550-20-1-0358 (Machine Learning and Physics-Based Modeling and Simulation). BH acknowledges support by the National Science Foundation grant number 
NSF-DMS-2208535 (Machine Learning for Bayesian Inverse Problems). HO also acknowedges support by the Department of Energy under award number DE-SC0023163 (SEA-CROGS: Scalable, Efficient and Accelerated Causal Reasoning Operators, Graphs and Spikes for
Earth and Embedded Systems).

\bibliographystyle{siamplain}
\bibliography{PDE-GP-Refs}

\appendix

\section{Sobolev Sampling Inequalities on Manifolds}\label{sec:AppA}
Below we collect useful sampling inequalities for Sobolev functions defined on smooth manifolds with corners. 
Following \cite[Chs.~1, 16]{lee-manifold} we consider a smooth, compact
Riemannian manifold $\M \subset \R^d$ of dimension $k \le d$ with corners, i.e., 
a Riemannian manifold with a smooth structure with corners; see \cite[Ch.~16]{lee-manifold}.
On such a manifold we define the natural geodesic distance 
\begin{equation*}
 \rho_\M: \M \times \M \to \R, \qquad   \rho_\M (x, y) := \inf \int_0^1 \| \dot{\ell}(t) \| dt,
\end{equation*}
where the infimum is taken over all piecewise smooth paths $\ell: [0,1] \mapsto \M$ 
satisfying the boundary conditions $\ell(0) = x$ and $\ell(1) = y$, and $\| \dot{\ell} (t)\|$ is 
the length of the tangent vector $\dot \ell (T)$ under the Riemannian metric. 

Following \cite{fuselier2012scattered} (see also \cite[Sec.~4.3]{taylor}) we further consider 
the Sobolev spaces  $H^k(\M)$ of functions defined on $\M$ as follows: 
Let $\mcl{A}= \{ M_j, \Psi_j\}_{j=1}^N$ be an atlas for $\M$ and let $\{ \kappa_j\}$ be 
a partition of unity of $\M$, subordinate to $M_j$. Then given functions $u: \M \to \R$ we define 
the Sobolev norms and the associated Sobolev spaces $H^s(\M)$ as 
\begin{equation*}
   H^s(\M):= \{ u: \M \to \R \mid \| u \|_{H^s(\M)} < + \infty \}, \qquad
    \| u \|_{H^s(\M)}:= \left( \sum_{j=1}^N \| \pi_j(u) \|_{H^s(\Xi_j)}^2 \right)^{1/2}, 
\end{equation*}
where the maps $\pi_j$ are defined as 
\begin{equation*}
    \pi_j(f) := 
    \left\{
    \begin{aligned}
    & \kappa_j ( f( \Psi_j^{-1}(y) ) ), &&\text{if  } y \in \Psi_j (M_j), \\
    & 0 && \text{otherwise.}
    \end{aligned}
    \right.
\end{equation*}
and the sets $\Xi_j$ are given by 
\begin{equation*}
        \Xi_j := \left\{
    \begin{aligned}
        & \R^k && \text{if $\Psi_j$ is an interior chart,}\\
        & \{ (x_1, \dots, x_k) \in \R^k | x_1 \ge 0 \} && \text{if $\Psi_j$ is a boundary chart,}\\ 
        & \{ (x_1, \dots, x_k) \in \R^k | x_1 \ge 0, \dots, x_k \ge 0 \} && \text{if $\Psi_j$ is a corner chart.}\\ 
    \end{aligned}
    \right.
\end{equation*}

Put simply, the Sobolev spaces $H^s(\M)$ are functions on $\M$ that, locally after 
the flattening of the manifold belong to the standard Sobolev spaces $H^s$.
With these notions at hand we then recall the following result of \cite{fuselier2012scattered},
which was proven by those authors for smooth embedded manifolds without boundary or corners. 
However, a brief investiation of the proof of that result reveals that it can 
immediately be generalized to our setting with manifolds with corners. In fact, the idea 
of the proof is to use the atlas to locally flatten the manifold and apply 
classic sampling theorems such as \cite[Thm.~4.1]{arcangeli2007extension} on each patch.
The only difference in the case of manifolds with corners is that the patches do not 
only map to $\R^k$ but rather to the subspaces $\Xi_j$ depending on whether 
the corresponding chart is an interior, boundary, or corner chart.



\begin{proposition}[{\cite[Lem.~10]{fuselier2012scattered}}] \label{fuselier-sobolev-bound}
Suppose $\M \subset \R^d$ is a smooth, compact, Riemannian manifold with corners, of dimension $k$ and let $s > k/2$ and $r \in \mathbb{N}$ satisfy
$0 \le r \le \lceil s \rceil - 1$. Let $X \subset \M$ be a discrete set with mesh norm $h_\M$
defined as 
\begin{equation*}
    h_\M := \sup_{x' \in \M}  \inf_{x \in X}  \rho_\M(x , x').
\end{equation*}
Then there is a 
constant $h_0>0$ depending only on $\M$ such that if $h_\M < h_0$ and if $u \in H^s(\M)$ satisfies 
$u|_X = 0$ then 
  \begin{equation*}
      \| u \|_{H^r(\M)} \le C h_\M^{s - r} \| u \|_{H^s(\M)}.
  \end{equation*}
Here $C >0$ is a constant independent of $h_\M$ and $u$.
\end{proposition}

\section{Bounds on Fill Distances}\label{sec:AppB}
\label{sec: Bounds on fill-in distance}
This section collects a result from \cite{reznikov2016covering} for bounding the fill-in distance for randomly distributed points on a manifold. 

Assume $(\M, \rho)$ is a metric space, and $\mu$ is a finite positive Borel measure supported on $\M$. Let $X = \{x_1,...,x_N\}$ be a set of $N$ points, independently and randomly drawn from $\mu$. Define the fill-in distance 
\begin{equation}
    h_\M = \sup_{x' \in \M} \inf_{x \in X} \rho(x,x')\, .
\end{equation}
Then, \cite[Thm. 2.1]{reznikov2016covering} implies the following:
\begin{proposition}
\label{prop: fill-in distance bound cited}
    Suppose $\Phi$ is a continuous non-negative strictly increasing function on $(0,\infty)$ satisfying $\Phi(r) \to 0$ as $r \to 0^{+}$. If there exists a positive number $r_0$ such that $\mu(B(x,r)) \geq \Phi(r)$ holds for all $x \in \M$ and every $r < r_0$, then there exist positive constants $c_1,c_2,c_3$ and $\alpha_0$ such that for any $\alpha>\alpha_0$, we have
    \begin{equation}
        \mathbb{P} \left[h_\M \geq c_1 \Phi^{-1}\left(\frac{\alpha \log N}{N}\right)\right]\leq c_2 N^{1-c_3\alpha}.
    \end{equation}
\end{proposition}
We use this proposition to prove \Cref{prop: fill-in dist bound}.
\begin{proof}[Proof of \Cref{prop: fill-in dist bound}]
We apply \Cref{prop: fill-in distance bound cited}. For the bounded domain $\Omega \subset \R^d$, we know that there exists a constant $C$ such that $\Phi(r) = C r^d$ will satisfy the assumption in \Cref{prop: fill-in distance bound cited}.
Moreover, we choose $\alpha$ such that $c_2M_{\Omega}^{1-c_3\alpha} \leq \delta$. This implies that
$\alpha \geq \frac{1}{c_3\log M_{\Omega}}\log(c_2M_{\Omega}/\delta)$. Pick $\alpha = \frac{C'}{c_3\log M_{\Omega}}\log(c_2M_{\Omega}/\delta)$ for some $C'\geq 1$ such that $\alpha \geq \alpha_0$. Then \Cref{prop: fill-in distance bound cited} shows that with probability at least $1-\delta$, \[h_{\Omega} \leq c_1 \Phi^{-1}\left(\frac{\alpha \log M_{\Omega}}{M_{\Omega}}\right) \leq C''\left(\frac{\log(M_{ \Omega}/\delta)}{M_{\Omega}}\right)^{1/d},  \]
where $C''$ is a constant independent of $M_{\Omega}$ and $\delta$. The bound on $h_{\partial \Omega}$ can be proved similarly by choosing $\Phi(r) = C r^{d-1}$.
\end{proof}

\section{The Choice of Nugget Terms}\label{sec:AppC}
For numerical stability, we add a diagonal adaptive nugget term to the kernel matrix in our computation such that
\[u^{\ell+1}(x) = K(x,\bphi^l)[K(\bphi^l,\bphi^l)+\eta\mathrm{diag}(K(\bphi^l,\bphi^l))]^{-1}\begin{pmatrix}
\left(f-u^{\ell}+\mP'(u^{\ell})u^{\ell}\right)|_{\bs_{\Omega}}\\
\left(g-u^{\ell}+\mB'(u^{\ell})u^{\ell}\right)|_{\bs_{\partial\Omega}}
\end{pmatrix}\]
Typically $\eta = 10^{-10}$. This nugget term is similar to the adaptive nugget term proposed in \cite{CHEN2021110668}. It is much more effective than the naive choice of $K(\bphi^l,\bphi^l)+\eta I$, since the conditioning of the interior block and the boundary block in the kernel matrix differs dramatically.

\end{document}